\documentclass[11pt, final]{article}
\usepackage[utf8]{inputenc}
\usepackage[english]{babel}
\usepackage{csquotes}
\usepackage{comment}
\usepackage{amsmath}
\usepackage{amsthm}
\usepackage{amssymb}
\usepackage{commath}
\usepackage{derivative}
\usepackage[mathcal]{eucal}
\let\mathscr\mathcal
\usepackage{color}
\usepackage{tensor}
\usepackage{breqn}
\usepackage{hyperref}
\usepackage[a4paper, total={6in, 9in}]{geometry}

\newtheorem*{remark}{Remark}

\newcommand{\supnorm}[1]{\left\lVert #1 \right\rVert_{\infty}}

\newcommand{\pnorm}[2]{\left\lVert #2 \right\rVert_{#1}}

\newtheorem{theorem}{Theorem}[section]
\newtheorem{corollary}{Corollary}[theorem]
\newtheorem{lemma}[theorem]{Lemma}

\title{$L^2$ Stability of Simple Shocks for Spatially Heterogeneous Conservation Laws}
\author{
	Shyam Sundar Ghoshal\thanks{Centre for Applicable Mathematics, Tata Institute of Fundamental Research, e-mail: ghoshal@tifrbng.res.in, venkatesh2020@tifrbng.res.in (corresponding author)} \and
	Parasuram Venkatesh\footnotemark[1]
}
\date{}
\begin{document}
\maketitle
\begin{abstract}
	In this paper, we consider scalar conservation laws with smoothly varying spatially heterogeneous flux that is convex in the conserved variable. We identify structural assumptions under which a single shock wave connecting two constant states emerges in finite time for all $L^{\infty}$ initial data satisfying the same far-field conditions. Under a further condition on the mixed partial derivative of the flux, we establish the $L^2$-stability of these simple shock profiles: perturbations of the Riemann initial data yield solutions whose $L^2$ distance from the corresponding simple shock wave is non-increasing in time, up to a time-dependent spatial shift. We further show that these conditions are sharp: we construct explicit counterexamples demonstrating that, in particular, the existence of a contractive shift function holds only under our assumptions and fails otherwise. The main tools we use are Dafermos' generalised characteristics for the evolution analysis and the relative entropy method for stability.
		
	\smallskip
		
	\noindent\textbf{Keywords:} scalar conservation laws, Hamilton--Jacobi equations, stability, relative entropy, generalised characteristics.
		
	\smallskip
		
	\noindent\textbf{MSC2020:} 35L65, 35L67 (Primary), 35A15 (Secondary)
\end{abstract}
\tableofcontents
\section{Introduction}
Consider the spatially heterogeneous scalar conservation law
\begin{equation}\label{claw}
	\begin{split}
		u_t+f(x,u)_x&=0, \\
		u(x,0)&=u_0(x).
	\end{split}
\end{equation}
Laws of this form arise as simple models of physical phenomena such as traffic flow on a road with smoothly varying maximal velocity \cite{thesis}. The heterogeneous ($x$-dependent) case is quite distinct in certain qualitative aspects from the homogeneous ($x$-independent) one. This has been of some recent interest, as seen in the work of Colombo, Perrollaz, and Sylla \cite{ConLawHJB}, which presents a mathematical framework for proving the global well-posedness of scalar conservation laws and Hamilton–Jacobi equations involving space-dependent variables using the vanishing viscosity and compensated compactness methods. The standard maximum principle fails for \eqref{claw}, so  the authors use a Bernstein method and a specially constructed family of stationary solutions to obtain uniform bounds in $L^{\infty}$ on viscous approximations. Existence of the limit without the $\operatorname{BV}$ bounds used in Kružkov's approach is guaranteed by compensated compactness. Crucially here the flux is required to be ``weakly genuinely non-linear'', i.e. the set of $u$ where $f_{uu}(x,\cdot)$ vanishes must have empty interior for almost every $x\in\mathbb{R}$.

Unlike the older theory of scalar conservation laws with discontinuous convex flux \cite{discofluxog}, the spatial heterogeneity in this article is assumed to be smooth.

In the case of spatially homogeneous and convex flux, characterisation of the entropy solution for data satisfying far-field conditions
\begin{equation}\label{sharefaction}
	u_0(x)=
	\begin{cases}
		u_{-}&\text{ if }x<x_{-}, \\
		u_{+}&\text{ if }x>x_{+}.
	\end{cases}
\end{equation}
was started in \cite{HomoShock} for the case of $u_->u_+$ with strictly convex flux, and completed in \cite{structuretheorem} for the general case.

In particular, Dafermos and Shearer \cite{HomoShock} proved that for initial data of the form \eqref{sharefaction}, a homogeneous $(f_x\equiv0)$ and strictly convex flux, and $u_{-}>u_{+}$, the entropy solution will reduce, in finite time, to a simple shock travelling at Rankine-Hugoniot speed. Surprisingly, the specific behaviour of $u_0$ in $[x_{-},x_{+}]$ is completely irrelevant except perhaps to delay the time and change the position from where the terminal shock forms. Adimurthi, Ghoshal, and Gowda \cite{structuretheorem} considered the other case, i.e. $u_{-}\leq u_{+}$ (allowing also for equality). The expectation is that the asymptotic shock in \cite{HomoShock} should now be an asymptotic rarefaction instead -- technically, one obtains convergence to the generalisation of an `N-wave' \cite[3.4.5b, pp. 157-161]{evans} with more intricate features. Unlike the shock case, solutions can demonstrate more non-trivial behaviours; the structure of entropy solutions is characterised primarily in terms of `Asymptotic Single Shock Packets' or ASSPs in the terminology of Adimurthi-Ghoshal-Gowda \cite{structuretheorem}, which are cylindrical and possibly tilted infinite portions of the space-time half-plane $\mathbb{R}^2_{+}$ within which all characteristics asymptotically merge into a single shock. The `speeds' of such cylinders are strictly increasing, from left to right, so they do not interact with each other. A fine analysis of backward characteristics plays a crucial role in the structure theorem. Additionally, they demonstrate that smooth ($C^{\infty}$) initial data can give rise to infinitely many such ASSPs.

In the first part of this paper, we tackle the first part of this characterisation program, i.e., the case of $u_->u_+$. In particular, we establish a finite-time emergence result in the vein of Dafermos and Shearer \cite{HomoShock} for a large class of initial data, demonstrating that all such data evolve into the same shock wave. This is the content of Theorem~\ref{SingleShock}. 

In the next part of the paper, we show that this shock profile generated by data of the form
\[
u_0(x)=
\begin{cases}
	u_{-}&\text{ if }x<0, \\
	u_{+}&\text{ if }x>0,
\end{cases}
\]
is stable with respect to $L^2$ perturbations. Thus, we recover the results of Leger \cite{LegerShock} in Theorem~\ref{L2 stability} for the heterogeneous case.

In his seminal paper, Kružkov \cite{Kruzkov} proved the well-posedness of the Cauchy problem; in particular, the unique `entropy' weak solutions of \eqref{claw} were shown to be characterised as those satisfying a further inequality
\begin{equation}\label{Het entropy}
	\eta(u)_t+Q(x,u)_x+\eta^{\prime}(u)f_x(x,u)-Q_x(x,u)\leq0,
\end{equation}
in the sense of distributions over $\mathscr{D}(\mathbb{R}\times[0,\infty))$ for all pairs of functions $\eta,Q$ such that the `entropy' $\eta$ is a convex function and the `entropy flux' $Q(x,\cdot)$ is an antiderivative of $\eta^{\prime}(\cdot)f_u(x,\cdot)$ for each fixed $x$. Note that given $\eta$, we can always construct $Q$ that satisfies this requirement.

For each such entropy-entropy flux pair $\eta,Q$ and smooth, non-negative $\varphi(x,t)$ compactly supported in $\mathbb{R}\times[0,\infty)$, an entropy-admissible weak solution must be such that
\[
\iint\eta(u)\varphi_t+Q(x,u)\varphi_x dxdt+\int\eta(u_0)\varphi(x,0)dx\geq\iint\varphi \left(\eta^{\prime}(u)f_x(x,u)-Q_x(x,u)\right)dxdt.
\]
Just as in the homogeneous case, this is motivated by the `vanishing viscosity' argument \cite{Kruzkov}. Important qualitative properties of solutions, such as the existence of boundary traces, are not directly established by the method, but when we consider non-degenerate fluxes, entropy solutions to \eqref{claw} have been shown to possess strong traces \cite{trace}. For convex fluxes in one spatial dimension, however, the notion of boundary trace can be more naturally obtained through the correspondence between scalar conservation laws and the Hamilton-Jacobi equations given by
\begin{equation}\label{HJ}
	\begin{split}
		v_t+f(x,v_x)&=0, \\
		v(x,0)&=v_0(x),
	\end{split}
\end{equation}
where the initial data $v_0$ is Lipschitz. For instance, Caselles \cite{correspondence} shows that this correspondence is exact, i.e. $v$ is a viscosity solution of \eqref{HJ} if and only if $u=v_x$ is an entropy solution of \eqref{claw}, when the flux has a stationary point at $0$. Formally differentiating the equation yields an $L^{\infty}$ solution $u=v_x$ of \eqref{claw}, with $L^{\infty}$ initial data $u_0=(v_0)_x$. This intuition is easily formalised in the homogeneous case \cite[pp. 143--148]{evans}, and a similar, albeit tedious, line of argument justifies this correspondence in the heterogeneous case as well \cite{thesis}. Since initial data for \eqref{claw} are typically $L^{\infty}$, it is natural to study Cauchy problems for \eqref{HJ} with Lipschitz initial data for the correspondence.

\subsection{Setting of the problems}
We assume that the flux $f:\mathbb{R}^2\to\mathbb{R}$ satisfies the following conditions:
\begin{flalign}
	\label{S}\tag{S}&\mbox{\textbf{Stationarity:}}\text{ for some }u^\pm \in \mathbb{R}, \ u^+ < u^- : f_x(\cdot, u^\pm) \equiv 0.&& \\
	\label{UC}\tag{UC}&\mbox{\textbf{Uniform Convexity:}}\text{ }f_{uu}\geq\alpha>0. \\
	\label{C2}\tag{$C^2$}&\mbox{\textbf{Smoothness:}}\text{ }f\in C^2(\mathbb{R}^2;\mathbb{R}). \\
	\label{Nag}\tag{N}&\mbox{\textbf{Nagumo Growth:}}\text{ there exists }N(\cdot)\in C([0,\infty);\mathbb{R}),\lim_{y\to\infty}\frac{N(y)}{y}=\infty\text{ such that } \\
	\notag&\text{for all } u\in\mathbb{R}:\abs{f(x,u)}\geq N\left(\abs{u}\right). \\
	\label{FSP}\tag{FSP}&\mbox{\textbf{Finite Speed of Propagation:}}\text{ }\theta(v)=\sup_{x\in\mathbb{R}}\abs{f_u(x,v)}\in C(\mathbb{R}).
\end{flalign}
If the flux satisfies these assumptions, then both \eqref{claw} and \eqref{HJ} are well-posed for $L^{\infty}$ and Lipschitz initial data, respectively, and moreover, the solutions correspond in the sense described above. In general, the correspondence holds even without \ref{S}, and \ref{UC} can be relaxed to strict convexity (see appendix~\ref{app}). The condition \eqref{Nag} is standard for the Hamilton-Jacobi equation (see \cite{CSHJB}) and ensures that the Legendre transform of the flux is uniformly well-defined. The role of \eqref{Nag} can be seen more explicitly in the discussion of well-posedness in the Appendix~\ref{app}, where the Legendre transform is also defined.

The results of Theorem~\ref{SingleShock} hold under these assumptions. For Theorem~\ref{L2 stability}, we require one further assumption on the heterogeneity of the flux (which is trivially satisfied for the homogeneous case), namely:
\begin{flalign}
	\label{P}\tag{P}\mbox{\textbf{Positive Heterogeneity:}}\text{ for all }y,z\in\mathbb{R}:f_{xu}(y,z)\geq0.&&
\end{flalign}
For an example of a flux that satisfies the assumptions {\eqref{S}-\eqref{FSP}}, consider an adaptation of the Lighthill-Whitham-Richards (LWR) model of traffic flow \cite{lwr}, in which the maximum velocity $V$ varies with $x$ and $V\in C^2(\mathbb{R})\cap L^{\infty}(\mathbb{R}),V\geq\alpha>0$. The conserved scalar in this case is the car density $u$ represented as a fraction of the maximum, and the flux given by
\[
f(x,u)=-V(x)u(1-u),
\]
clearly satisfies all the assumptions \ref{S}-\ref{UC}-\ref{C2}-\ref{Nag}-\ref{FSP} with $u_{+}=0,u_{-}=1$. This example is adapted from \cite{thesis}.

In the traffic flow model, of course, the flux is concave, not convex, i.e. {the flux is} $-f$ rather than $f$, but this only amounts to a change of direction from $x$ to $-x$ and does not substantially affect the analysis. Now suppose $f$ is as above and initial data $u_0(x)=H(-x)$, where $H$ is the Heaviside function, i.e.,
\[
H(x)=
\begin{cases}
	0&\text{ if }x<0, \\
	1&\text{ if }x>0.
\end{cases}
\]
Then, by \ref{S} and \ref{UC} the entropy solution to the Cauchy problem \eqref{claw} is the simple shock $u(x,t)=H(-x)$, as if this were a homogeneous conservation law.

The motivation behind this model is that traffic density ranges from a minimum of zero to some maximum, and the velocity of a car in the road varies inversely with it. At the maximum density, cars slow to a halt and the velocity, hence the flux, must be zero. Similarly, the flux must be zero when the density is as well, which motivates the concave shape of the flux function $f$. Under this interpretation, the equation is thus akin to the conservation of mass. Allowing the maximum velocity to vary with $x$ allows us to model heterogeneous road conditions.

This phenomenon generalises as long as the flux satisfies \ref{S}. Let $f(u_{\pm})$ denote the constant values that $f(\cdot,u_{\pm})$ attain respectively. For initial data $u_0(x)\equiv u_{\pm}$, it is clear that $u(x,t)\equiv u_{\pm}$ is the unique entropy solution. Given the spatial heterogeneity of $f$, these are, in general, the only constant admissible solutions, in contrast with the homogeneous case where all constant functions on the $(x,t)$ upper half plane are entropy solutions of \eqref{claw}. Thus, for Riemann-type initial data of the form
\begin{equation}\label{Riemann}
	\Phi(x)=
	\begin{cases}
		u_{-}&\text{ if }x<0, \\
		u_{+}&\text{ if }x>0,
	\end{cases}
\end{equation}
where $u_{-}>u_{+}$, the entropy solution to \eqref{claw} is given by $u(x,t)=\Phi(x-\sigma t)$, where $\sigma$ is the Rankine-Hugoniot shock speed given by
\begin{equation}\label{RH}
	\sigma=\dfrac{f(u_{-})-f(u_{+})}{u_{-}-u_{+}}.
\end{equation}
Here $f(u_{\pm})$ denote the constant values that $f$ takes at $u_{\pm}$.

Apart from models like LWR for traffic flow, one major motivation for studying scalar conservation laws is the hope of discovering methods that generalise to systems. For instance, the idea of generalised characteristics \cite{GenCharSys} or the theory of $L^2$ stability of shock waves for hyperbolic systems (the method of `a-contraction with shifts') as pioneered by Golding, Krupa, and Vasseur \cite{sharpL2contra}.

The latter has been fruitfully extended to improve uniqueness results for the general Cauchy problem \cite{bvuni}. Similar extensions for heterogeneous systems have not yet been pursued to the same extent, but the hope is that such a program can be carried out.

Typically, the methods do not carry over completely and with full power, but often serve to motivate. For instance, the theory of contraction up to shift for scalar conservation laws as in \cite{LegerShock} only holds for specific systems, and the method must be modified by weighting the relative entropy term to remain valid for systems. Even then, the contraction property may not {hold} for non-extremal shocks. Some sharp criteria for the existence of contractive shifts are presented in \cite{criteriacontra}.
\begin{remark}
	In place of {\ref{S}}-\ref{FSP}, the correspondence theory in \cite{ConLawHJB} instead assumes
	\begin{flalign}
		\label{CNH}\tag{CNH}\mbox{\textbf{Compact Non-Homogeneity:}}\text{ there exists }A>0:\forall\abs{x}\geq A:f_x(x,\cdot)\equiv0.&&
	\end{flalign}
	From \ref{CNH}, the properties \ref{Nag} and \ref{FSP} can easily be derived. Here, we assume \ref{Nag} directly to obtain the existence and boundedness of minimising arcs for the Hamilton-Jacobi equation in its variational form \cite{CSHJB}, and \ref{FSP} so that we can work locally. {This establishes the existence of viscosity solutions to \eqref{HJ}, and thus we can conclude that entropy solutions to \eqref{claw} exist.}
\end{remark}
\begin{remark}
	As we shall see in the proof, $L^2$ is special among the $L^p$ spaces since the norm is \textit{uniformly} convex, i.e., $f(u)=u^2$ is such that $f^{\prime\prime}=1$, whereas for $p>2$ and $g(u)=\abs{u}^p$, we have that $g^{\prime\prime}(u)=p(p-1)\abs{u}^{p-2}$, which vanishes at zero.
	
	If $p<2$, then since we are dealing with $L^{\infty}$ initial data, any $L^p$ perturbation is, by interpolation, also an $L^2$ perturbation. Thus, we only restrict ourselves to the $L^2$ case. Even for homogeneous scalar conservation laws, similar results for $p\neq2$ require further restrictions on the initial perturbations \cite{lp}.
\end{remark}
\begin{remark}\label{band remark}
	We require \ref{P} to cast the entropy inequality \eqref{Het entropy} in divergence form, as we shall see later. However, \ref{P} along with \ref{S} is more restrictive than it appears; in particular it implies that $f_x(\cdot,c)\equiv0$ for all $c\in[u_{+},u_{-}]$, not just the endpoints. Thus, the flux is effectively homogeneous in this band of values; this observation plays a crucial role in the proof of Lemma~\ref{lambda}. However, refer to section~\ref{neg het} to see why \eqref{P} is quite sharp with respect to Theorem~\ref{L2 stability}. Thus, we do not expect that it can be relaxed.
\end{remark}

\subsection{Example of an admissible flux}\label{ex}
Let us construct an explicit example of a flux satisfying all the stipulated conditions, including \eqref{P}. The heterogeneous LWR flux is presented to motivate the physical relevance of spatial heterogeneity. However, it does not satisfy \eqref{P}. To see an example of a flux that does, we present the following construction. Let $h, g \in C^2(\mathbb{R})$ be uniformly convex with common convexity constant $\alpha > 0$ such that $h(0)=g(0)$ and $h^{\prime}(u)\geq g^{\prime}(u)$, with equality holding if $u\in[0,1]$ but $h\neq g$ in general. Now, we define a spatially heterogeneous flux using $h,g$. Let $\varphi(x)$ be a monotone increasing, smooth function with $0\leq\varphi\leq1$, e.g. a regularisation of the Heaviside function $H$ by convolution with a symmetric mollifier. Then define
\[
f(x,u)=\varphi(x)h(u)+(1-\varphi(x))g(u).
\]
Then, since $f\in C^2(\mathbb{R}^2)$,
\begin{align*}
	f_{xu}(x,u)&=\pdv{}{x}\left(\varphi(x)(h^{\prime}(u)-g^{\prime}(u))+g^{\prime}(u)\right) \\
	&=\varphi^{\prime}(x)(h^{\prime}(u)-g^{\prime}(u)) \\
	&\geq0,
\end{align*}
by our hypotheses on $h,g,\varphi$. The other conditions are easily verified.
\begin{itemize}
	\item \eqref{S}: since $h(0)=g(0)=0$, and by construction $h^{\prime}=g^{\prime}$ on the interval $[0,1]$, we have that $h(1)=g(1)$ as well, and stationarity holds with $u_+=0,u_-=1$.
	\item \eqref{UC}: we can explicitly compute $f_{uu}(x,u)=\varphi(x)h^{\prime\prime}(u)+(1-\varphi(x))g^{\prime\prime}(u)\geq\alpha$.
	\item \eqref{C2}: since $\varphi,h,g$ are smooth, $f$ is trivially $C^2$.
	\item \eqref{Nag}: by uniform convexity, $h,g$ satisfy Nagumo growth, say with functions $N_h,N_g$ respectively. Then, for $N(y)=\min\{N_h(y),N_g(y)\}$, we have that $N$ is a Nagumo function for $f$.
	\item \eqref{FSP}: again we explicitly compute
	\begin{align*}
		\theta(v)&=\sup_x\abs{f_u(x,v)} \\[1.5ex]
		&=\sup_x\abs{\varphi(x)h^{\prime}(v)+(1-\varphi(x))g^{\prime}(v)} \\[1.5ex]
		&=\max\left\{\abs{h^{\prime}(v)},\abs{g^{\prime}(v)}\right\}.
	\end{align*}
	This is a continuous function in $v$.
\end{itemize}
For an explicit example, let $a(u)=u$, and $b(u)=\eta_{\epsilon}\ast\tilde{b}(u)$, where $\eta_{\epsilon}$ is a standard symmetric mollifier supported in $[-\epsilon,\epsilon]$ with $0<\epsilon<1/4$, and
\[
\tilde{b}(u)=
\begin{cases}
	-\dfrac{1}{2}+\dfrac{3}{2}\left(u+\dfrac{1}{2}\right)&\text{ if }u<-\dfrac{1}{2}, \\
	u&\text{ if }-\dfrac{1}{2}\leq u<\dfrac{3}{2}, \\
	\dfrac{3}{2}+\dfrac{2}{3}\left(u-\dfrac{3}{2}\right)&\text{ if }u\geq\dfrac{3}{2}.
\end{cases}
\]
Clearly $a,b$ are both monotonically increasing, smooth functions with $a\geq b$, with equality holding at least on the interval $[0,1]$. Hence, if we define
\[
h(u)=\int_{0}^{u}a(y)dy\text{ and }g(u)=\int_{0}^{u}b(y)dy,
\]
then $h,g$ satisfy all our requirements.

\section{Main Results}\label{MR}
First, we consider initial data that satisfy the following far-field condition as in \cite{HomoShock}: suppose there exist real numbers $x_{\pm}\in\mathbb{R}, x_{-}<x_{+}$ such that:
\begin{equation}\label{far field conditions}
	u_0(x)=
	\begin{cases}
		u_{-}&\text{ if }x<x_{-}, \\
		u_{+}&\text{ if }x>x_{+}.
	\end{cases}
\end{equation}
For such initial data, we have the following result for entropy solutions.
\begin{theorem}\label{SingleShock}
	\textbf{(Finite-time emergence of a simple shock)} Let $f$ satisfy the assumptions \ref{S}-\ref{FSP}, and let $u(x,t)$ be the entropy solution of the Cauchy problem \eqref{claw} with initial data $u_0$ satisfying far-field conditions \eqref{far field conditions}. Then, for some $X\in\mathbb{R}$ and finite $T>0$, we have that for all $t>T$:
	\[
	u(x,t)=
	\begin{cases}
		u_{-}&\text{ if }x<X+\sigma(t-T), \\
		u_{+}&\text{ if }x>X+\sigma(t-T).
	\end{cases}
	\]
	Here $\sigma$ is the shock speed coming from the Rankine-Hugoniot condition \eqref{RH}.
\end{theorem}
We remark that for spatially homogeneous fluxes, this result was proved by Dafermos and Shearer \cite{HomoShock}. With the additional constraint \ref{P} on the flux, we have the following result.
\begin{theorem}\label{L2 stability}
	\textbf{(Stability of the simple shock)} Let $f$ satisfy the assumptions \ref{S}-\ref{FSP}, and consider the Cauchy problem \eqref{claw} with initial data $u_0\in L^{\infty}$ such that $\pnorm{2}{u_0-\Phi}<\infty$, where $\Phi$ is defined as in \eqref{Riemann}. Further suppose $f$ satisfies \ref{P} in a compact set containing $[u_+,u_-]$ as well as $[\operatorname{ess\text{ }inf} u,\operatorname{ess\text{ }sup} u]$, where $u$ is the entropy solution to \eqref{claw} with initial data $u_0$ and the extrema are taken globally in space-time. Then, there exists a Lipschitz curve $\overline{\xi}$ such that, for all $t>0$:
	\[
	\pnorm{2}{u(\cdot,t)-\Phi(\cdot-\overline{\xi}(t)-\sigma t)}\leq\pnorm{2}{u_0-\Phi}.
	\]
	Note that if $\overline{\xi}(t)=0$, then $v(x,t)=\Phi(x-\sigma t)$ is the entropy solution of \eqref{claw} with initial data \eqref{Riemann}. Thus $\overline{\xi}$ represents the magnitude of the shift required to generate a contraction with respect to the $L^2$-norm.
\end{theorem}
The magnitude of the shift $\overline{\xi}$ can be controlled by the magnitude of the initial perturbation. Thus, in particular, we have the following corollary.
\begin{corollary}\label{magnitude}
	\textbf{(Magnitude of the shift)} For any curve $\overline{\xi}$ as in Theorem \ref{L2 stability}, we claim that $\abs{\overline{\xi}(t)}\leq K\sqrt{t}$, where the constant $K$ depends only on $\supnorm{u_0},f,u_{\pm},\text{ and }\pnorm{2}{u_0-\Phi}$.
\end{corollary}
The fact that $\overline{\xi}$ is Lipschitz is \textit{not} established by this corollary; it is proved independently in Theorem~\ref{L2 stability}.

\section{Preliminaries}
We briefly recall Dafermos' theory of `generalised characteristics' for scalar conservation laws with convex flux \cite{GenChar}, which we use throughout the paper.

\subsection{Generalised characteristics}
For smooth solutions of \eqref{claw}, the system of characteristic ODEs is as follows.
\begin{equation}\label{char ode}
	\begin{split}
		\dot{y}(s)&=f_u(y(s),z(s)), \\
		\dot{z}(s)&=-f_x(y(s),z(s)),
	\end{split}
\end{equation}
where $y(s)$ is the trajectory of the characteristic in space-time, while the paired function $z(s)$ corresponds to the value of the solution $u$ along the curve $y(s)$. More generally, given an entropy solution of \eqref{claw}, we have that from every point $(x,t)$ with $t>0$, we can define a unique \textit{forward characteristic} $y_f:[t,\infty)\to\mathbb{R}$ and a non-empty set of \textit{backward characteristics} $y_b:[0,t]\to\mathbb{R}$, i.e. Lipschitz curves with $y(t)=x$ solving the differential inclusion
\begin{equation}\label{DafDiff}
	\dot{y}(s)\in[f_u(y(s),u(y(s)+,s)),f_u(y(s),u(y(s)-,s))]
\end{equation}
on their respective domains, where $u(x\pm,t)$ respectively denote the left and right traces in space of $u$ at $(x,t)$. Since $f$ is convex in the second variable, $f_u$ is monotonically increasing, and entropy solutions of $\eqref{claw}$ satisfy the inequality $u(y(s)-,s)\geq u(y(s)+,s)$, hence the interval in \eqref{DafDiff} is well-defined for all $(x,t)$. Forward characteristics are also defined for the initial time when left and right traces exist, but they may not be unique.

This differential inclusion is also stronger than it appears prima facie. At points of continuity of $u$, of course, $\dot{y}$ only has one permissible value, but even on points of discontinuity, it can be shown that $\dot{y}$ has a determinate value. That is, let $y:[t_0,T]\to\mathbb{R}$ be a Lipschitz solution to \eqref{DafDiff} for some $t_0\geq0$. Then, for almost all $t\in[t_0,T]$, we have that
\[
\dot{y}(t)=
\begin{cases}
	&f_u(y(t),u(y(t),t))\text{ if }u(y(t)-,t)=u(y(t)+,t), \\
	\smallskip
	&\dfrac{f(y(t),u(y(t)-,t))-f(y(t),u(y(t)+,t))}{u(y(t)-,t)-u(y(t)+,t)}\text{ if }u(y(t)-,t)>u(y(t)+,t).
\end{cases}
\]
The second equation of \eqref{char ode} now comes about as the value function. Indeed, the characteristics that only pass through points of continuity of $u$, termed as \textit{genuine characteristics}, satisfy $\eqref{char ode}$ with $z(t)=u(y(t),t)$. The other alternative, satisfied at points of discontinuity, is just the Rankine-Hugoniot condition. Such characteristics are called \textit{shocks}. A forward characteristic that is a shock at some time $T$, moreover, remains a shock for all greater times as well. That is to say, shocks cannot vanish in finite time, though they may merge. Thus, a forward characteristic from a point of discontinuity is always a shock.

The set of backward characteristics from a point $(x,t)$ is bounded by two curves $y_{\pm}$, which we call the \textit{minimal} and \textit{maximal} characteristics, respectively. These two (possibly equal) curves are themselves genuine characteristics in the sense defined above, and solve \eqref{char ode} along with $z\pm(s)=u(y_\pm(s),s)$ with the terminal conditions
\[
\begin{split}
	y_\pm(t)&=x, \\
	z_\pm(t)&=u(x\pm,t).
\end{split}
\]
Thus, given an entropy solution $u$, every point in space at positive time has at least one backward characteristic associated with it, which is unique if and only if $x$ is a point of continuity of $u$. Moreover, the value of the \textit{cádlag} {(right-continuous with left limits)} representative of $u$ is given by the value $z(t)$ associated with its minimal backward characteristic at $t$. The left limit of $u$ on the other hand, gives us the value $z(t)$ associated with the maximal backward characteristic. One other important fact about genuine characteristics is that they only intersect at their endpoints.

In the homogeneous case, it is easily seen that such characteristics/minimisers must be straight lines in space-time. This yields the well-known Hopf-Lax and Lax-Oleinik formulas for solutions of \eqref{HJ} and \eqref{claw} respectively \cite{evans}. Another advantage of this `correspondence method' is that qualitative properties of solutions, and also some control aspects, become easier to analyse \cite{HomoShock,structuretheorem,exactcontrol}.

\subsection{Structure of the paper}
This article has been arranged as follows: in section \ref{MR} we state the theorems proved in our framework; the proofs themselves are deferred to sections \ref{proofs} and \ref{proofs2}. There is a subtle back-and-forth in the overall argument establishing the results. In particular, a special case of Theorem~\ref{SingleShock} is used to prove the general case of Lemma~\ref{main lemma}, which in turn proves Theorem~\ref{SingleShock}, without arguing in a circle.

We remark that Lemma~\ref{main lemma} comprises the bulk of our argument with respect to Theorem~\ref{SingleShock}, and lies at the heart of this paper. Since the comparison principle still holds for \eqref{claw} even with spatial heterogeneity, the argument proving Theorem~\ref{SingleShock} from Lemma~\ref{main lemma} is quite straightforward. The major difficulty that we face is establishing the finite-in-time emergence result for the special initial data as detailed in the lemma, for which we require the assumptions \ref{S}-\ref{FSP}.

Let us explicitly detail the steps for reference, so as to show that our argument is not circular.
\begin{enumerate}
	\item We show that Theorem~\ref{SingleShock} follows from Lemma~\ref{main lemma}.
	
	\item We prove Lemma~\ref{main lemma} in the special case where $u_{m}=u_{+}$ and $u_{M}=u_{-}$; this will also prove Theorem~\ref{SingleShock} for initial data $u_0$ taking values in $[u_{+},u_{-}]$ and satisfying the far-field conditions \eqref{far field conditions}.
	
	\item We show that the general case of Lemma~\ref{main lemma} reduces to the special case in finite time, i.e., for some $T>0$, we show that $u(x,T)\in[u_{+},u_{-}]$ for almost all $x\in\mathbb{R}$.
	
	\item This completes the proof of Lemma~\ref{main lemma} by the semigroup property of solutions to \eqref{claw}, since we have already established Theorem~\ref{SingleShock} for the special case.
	
	\item Hence, this also completes the proof of Theorem~\ref{SingleShock} in full generality, since we have already shown that it follows from Lemma~\ref{main lemma}.
\end{enumerate}
Then, under a further assumption on the mixed partial derivatives of the flux, as detailed below, we prove Theorem~\ref{L2 stability} and Corollary~\ref{magnitude}, establishing quantitative estimates on the stability of our shock profile connecting $u_{-}$ to $u_{+}$.

\section{Proofs I}\label{proofs}
\subsection{Finite-time emergence of a simple shock}
We prove Theorem~\ref{SingleShock} for a specific class of piecewise constant initial data in the following lemma, and prove the general case by a comparison argument. The proof of the lemma is deferred to a later subsection.
\begin{lemma}\label{main lemma}
	\textbf{(Evolution of piecewise constant initial data)} Let $f$ be a flux satisfying the assumptions \ref{S}-\ref{FSP}, and let $u(x,t)$ be the entropy solution of the Cauchy problem \eqref{claw} with initial data $u_0$ such that for some $x_{-}\leq x_0\leq x_{+}$:
	\begin{equation}\label{special data}
		u_0(x)=
		\begin{cases}
			u_{-}&\text{ if }x<x_{-}, \\
			u_m&\text{ if }x\in(x_{-},x_0), \\
			u_M&\text{ if }x\in(x_0,x_{+}), \\
			u_{+}&\text{ if }x>x_{+},
		\end{cases}
	\end{equation}
	where $u_m\leq u_{+}<u_{-}\leq u_M$. Then, for some $X\in\mathbb{R}$ and finite $T>0$, we have that for all $t>T$:
	\[
	u(x,t)=
	\begin{cases}
		u_{-}&\text{ if }x<X+\sigma(t-T), \\
		u_{+}&\text{ if }x>X+\sigma(t-T).
	\end{cases}
	\]
\end{lemma}
That is to say, the entropy solution $u$ exhibits a simple shock from the point $(X,T)$.

Assuming the validity of Lemma~\ref{main lemma}, let us proceed to prove the theorem. This proof follows the argument of \cite{HomoShock}, but we provide the details here for completeness. The non-trivial effect of heterogeneity enters in the proof of Lemma~\ref{main lemma}. Although the backward characteristics are not necessarily straight lines in space-time as in the homogeneous case, this does not affect the argument much. Let $u_0\in L^{\infty}(\mathbb{R})$ satisfying the far-field conditions \eqref{far field conditions} be given. Define $u_m=\operatorname{ess\text{ }inf}{u_0}$ and $u_M=\operatorname{ess\text{ }sup}{u_0}$; by construction $u_m\leq u_{+}<u_{-}\leq u_M$, by \eqref{far field conditions}. Now let $x_0\in[x_{-},x_{+}]$ be the unique value such that
\begin{equation}\label{convex combo}
	\int_{x_-}^{x_+}u_0(x)dx=(x_0-x_-)u_m+(x_+-x_0)u_M,
\end{equation}
and define $\overline{u}_0(x)$ as follows:
\[
\overline{u}_0(x)=
\begin{cases}
	u_{-}&\text{ if }x<x_{-}, \\
	u_m&\text{ if }x\in(x_{-},x_0), \\
	u_M&\text{ if }x\in(x_0,x_{+}), \\
	u_{+}&\text{ if }x>x_{+},
\end{cases}
\]
so that by \eqref{convex combo},
\[
\int_{x_-}^{x_+}u_0(x)dx=\int_{x_{-}}^{x_{+}}\overline{u}_0(x)dx.
\]
Now, from Lemma~\ref{main lemma} we know that the entropy solution to the scalar conservation law with initial data $\overline{u}_0(x)$ evolves into a simple shock in finite time. Let $\overline{u}(x,t)$ be the unique entropy solution of the conservation law for initial data $\overline{u}_0$ and $u(x,t)$ the unique entropy solution for initial data $u_0$; we have that
\begin{equation}
	\label{divzero}
	(\overline{u}-u)_t+(f(x,\overline{u})-f(x,u))_x=0.
\end{equation}
Let $(X,T)$ denote a point satisfying the following properties:
	\begin{itemize}
		\item for $t>T$, the solution $\overline{u}(x,t)$ is a simple shock connecting $u_-$ and $u_+$, with the shock curve passing through $(X,T)$ with speed $\sigma$.
		\item If $y_{-},y_{+}$ denote the minimal and maximal backward characteristics from $(X,T)$ corresponding to $\overline{u}$, then $y_{-}(0)<x_-,y_+(0)>x_+$.
	\end{itemize}
We claim that $u$ also exhibits a simple $\sigma$-shock from this point. That is, $u=\overline{u}$ for $t>T$. To see this, pick $\hat{t}>T$ and let $\hat{x}<X+\sigma(\hat{t}-T)$ be a point of continuity for $u$. Since the flux is uniformly convex in the $u$-variable, a unique backward characteristic exists. Let, if possible, $u(\hat{x},\hat{t})\neq u_{-}$. We know that there is a unique backward characteristic from $(\hat{x},\hat{t})$ for $\overline{u}$ corresponding to the value $u_{-}$. Let $y,\overline{y}$ respectively denote these characteristics, and $z,\overline{z}$ the respective value functions along these characteristics. We already know that $\overline{z}\equiv u_{-}$.

Now, by \cite[Theorem 3.3]{GenChar}, if $y(0)<x_{-}$, then $z(0)=u_{-}$ and hence by uniqueness of solutions to the ODE system \eqref{char ode} for initial conditions $(y(0),z(0))=(y_0,u_{-})$, we would have $u(\hat{x},\hat{t})=z(\hat{t})=u_{-}$. Therefore, suppose if possible that $y_0=y(0)\geq x_{-}$, where we include $x_-$ because the initial data may be discontinuous at this point, and the initial value of $z$ may not be determined merely by $y(0)$; in this case we can effectively apply the `divergence theorem' to the region enclosed by the two curves \cite{GenChar}; then by \eqref{divzero} and the characteristic equations \eqref{char ode}:
\[
\begin{split}
	0\leq&\int_{\overline{y}(0)}^{y(0)}u_0(x)-\overline{u}_0(x)dx \\
	=&\int_{0}^{\hat{t}}f(\overline{y}(t),\overline{z}(t))-f(\overline{y}(t),u(\overline{y}(t),t))-f_u(\overline{y}(t),\overline{z}(t))(\overline{z}(t)-u(\overline{y}(t),t))dt \\
	&-\int_0^{\hat{t}}f(y(t),\overline{u}(y(t),t))-f(y(t),z(t))-f_u(y(t),z(t))(\overline{u}(y(t),t)-z(t))dt, \\
\end{split}
\]
where the first inequality follows because $\overline{y}(0)<x_{-}$ by the non-intersection property of genuine characteristics, and the integral is constructed precisely to ensure that it is well-signed by the choice \eqref{convex combo} of $x_0$. On the other hand, by \ref{UC}, we have that
\[
f(x,b)-f(x,a)-f_u(x,a)(b-a)\geq\dfrac{\alpha}{2}(b-a)^2,
\]
therefore if $u(\hat{x},\hat{t})\neq u_-$, the integrands over $[0,\hat{t}]$ are non-trivial by continuity at least near $t=\hat{t}$, and hence (note that the integrals appear with a negative sign, flipping the inequality we obtain from~\ref{UC})
\[
\begin{split}
	0=&-\int_{0}^{\hat{t}}f(\overline{y}(t),u(\overline{y}(t),t))-f(\overline{y}(t),\overline{z}(t))-f_u(\overline{y}(t),\overline{z}(t))(u(\overline{y}(t),t)-\overline{z}(t))dt \\
	&-\int_0^{\hat{t}}f(y(t),\overline{u}(y(t),t))-f(y(t),z(t))-f_u(y(t),z(t))(\overline{u}(y(t),t)-z(t))dt \\
    \leq&-\dfrac{\alpha}{2}\int_{0}^{\hat{t}}\left(u(\overline{y}(t),t)-\overline{z}(t)\right)^2dt-\dfrac{\alpha}{2}\int_0^{\hat{t}}\left(\overline{u}(y(t),t)-z(t)\right)^2dt \\
	<&\text{ }0,
\end{split}
\]
which is a contradiction. Therefore, our assumption $u(\hat{x},\hat{t})\neq u_{-}$ is false. Similarly we can show that for $\hat{x}>X+\sigma(\hat{t}-T),u(\hat{x},\hat{t})=u_{+}$ almost everywhere, which completes the proof. Now that we see how the main theorem follows from Lemma~\ref{main lemma}, let us proceed to prove the lemma itself.

\subsection{Evolution of piecewise constant initial data}\label{mainlemma proof}
Lemma~\ref{main lemma} is proved in two steps; first, we prove the lemma for initial data satisfying $u_0(x)\in[u_{+},u_{-}]$ with $u_M=u_{-},u_m=u_{+}$ as per the earlier definitions, and then proceed to the general case. In both cases, the initial data has a shock-type discontinuity at $x=x_{\pm}$. Let $s_{\pm}$ denote the shock curves emanating from $x_{\pm}$ respectively.

Note that if $x_0=x_{\pm}$, then there is nothing to prove, since the initial data is of the simple Riemann problem form, and a single shock emerges from $t=0$ itself from $x_0$. Hence, in what follows, we assume that $x_0$ lies strictly between $x_-$ and $x_+$.

If a generalised characteristic is a shock at some time, then it remains a shock for all greater times. By the non-intersection property of genuine characteristics, it is enough to show that the two shock curves meet in finite time. By the Rankine-Hugoniot conditions:
\[
\dot{s}_{+}(t)=\dfrac{f(s_{+}(t),u_l(s_{+}(t),t))-f(u_+)}{u_l(s_{+}(t),t)-u_{+}},
\]
\[
\dot{s}_{-}(t)=\dfrac{f(u_-)-f(s_{-}(t),u_r(s_{-}(t),t))}{u_{-}-u_r(s_{-}(t),t)},
\]
where $u_l,u_r$ denote the left and right traces of the solution at point(s) on the shock curve(s). Now, note that along genuine characteristics, $f$ is constant, since from the characteristic equations \eqref{char ode}  and \ref{C2} we have:
\[
\begin{split}
	\frac{d}{dt}f(x(t),z(t))&=f_x(x(t),z(t))\dot{x}(t)+f_u(x(t),z(t))\dot{z}(t) \\
	&=f_x(x(t),z(t))f_u(x(t),z(t))-f_u(x(t),z(t))f_x(x(t),z(t)) \\
	&=0.
\end{split}
\]
Hence, as the shocks $s_{\pm}$ begin to interact with the rarefaction from $x_0$, the traces $u_l, u_r$ lie in $(u^+, u^-)$ since, by the Cauchy-Lipschitz theory for ODEs, orbits of the characteristic equations cannot cross, so by \ref{UC} we have that, in any time interval where the shocks do not meet:
\[
\dot{s}_+(t)=\dfrac{f(s_{+}(t),u_l(s_{+}(t),t))-f(u_+)}{u_l(s_{+}(t),t))-u_{+}}<\frac{f(u_-)-f(u_+)}{u_--u_+}=\sigma,
\]
\[
\dot{s}_-(t)=\dfrac{f(u_-)-f(s_{-}(t),u_r(s_{-}(t),t))}{u_{-}-u_r(s_{-}(t),t)}>\frac{f(u_-)-f(u_+)}{u_--u_+}=\sigma.
\]
Hence, once both shocks begin interacting with the rarefactions, and before they meet (if at all), $\dot{s}_+(t)-\dot{s}_-(t)<0$ almost everywhere in $t$. Now, this is not sufficient to guarantee intersection; we need a uniformly negative lower bound for $\dot{s}_+-\dot{s}_-$. For this, we employ \ref{UC} again. Firstly, since $u_l,u_r\in(u_{+},u_{-})$, there exist functions $\lambda_{l}(t),\lambda_{r}(t)$ taking values in $(0,1)$ such that
\[
\begin{split}
	u_l(s_{+}(t),t)&=u_{+}+\lambda_{l}(u_{-}-u_{+}), \\
	u_r(s_{-}(t),t)&=u_{-}+\lambda_{r}(u_{+}-u_{-}).
\end{split}
\]
Now consider $u_l$; a similar analysis can be carried out for $u_r$. By \ref{UC}:
\[
f(s_{+}(t),u_{+}+\lambda_{l}(u_{-}-u_{+}))-f(u_{+})\leq\lambda_{l}(f(u_{-})-f(u_{+}))-\frac{\alpha}{2}\lambda_{l}(1-\lambda_{l})(u_{-}-u_{+})^2.
\]
Hence,
\[
\begin{split}
	\dot{s}_{+}(t)&=\dfrac{f(s_{+}(t),u_{+}+\lambda_{l}(u_{-}-u_{+}))-f(u_{+})}{\lambda_{l}(u_{-}-u_{+})} \\
	&\leq\sigma-\frac{\alpha}{2}(1-\lambda_{l})(u_{-}-u_{+}).
\end{split}
\]
Similarly, we can derive that
\[
\dot{s}_{-}(t)\geq\sigma+\frac{\alpha}{2}(1-\lambda_{r})(u_{-}-u_{+}).
\]
Now, suppose $f(u_{-})>f(u_{+})$; let $t_0$ denote the time when $s_{+}$ starts interacting with the rarefaction, and let $\delta>0$ be small enough so that $s_{\pm}$ do not interact in $(t_0,t_0+\delta)$. Now,
\[
f(u_{-})>f(s_{+}(t_0+\delta),u_{l}(s_{+}(t_0+\delta),t_0+\delta)).
\]
Let $\epsilon=f(u_{-})-f(s_{+}(t_0+\delta),u_{l}(s_{+}(t_0+\delta),t_0+\delta))$, so $\epsilon>0$. Since $f$ is constant along characteristics, we have that for all $t>t_0+\delta:$
\[
f(s_{+}(t),u_l(s_{+}(t),t))<\max{(f(u_{-})-\epsilon,f(u_{+}))},
\]
assuming the two shock curves $s_{\pm}$ do not interact. Then, we have that
\[
\begin{split}
	&f(s_{+}(t),u_{l}(s_{+}(t),t))-f(u_{+}) \\
	\leq&\min\left\{f(u_{-})-f(u_{+})-\epsilon,\lambda_{l}(f(u_{-})-f(u_{+}))-\frac{\alpha}{2}\lambda_{l}(1-\lambda_{l})(u_{-}-u_{+})^2\right\}.
\end{split}
\]
Hence, dividing by $\lambda_{l}(u_{-}-u_{+})$, we have that for all $t>t_0+\delta$:
\[
\dot{s}_{+}(t)\leq\min\left\{\frac{\sigma}{\lambda_{l}}-\frac{\epsilon}{\lambda_{l}(u_{-}-u_{+})},\sigma-\frac{\alpha}{2}(1-\lambda_{l})(u_{-}-u_{+})\right\}.
\]
Hence, $\dot{s}_{+}-\sigma$ is uniformly bounded above by some negative number, whether $\lambda_{l}$ is close to one or not. A similar analysis can be carried out for the case of $f(u_{-})\leq f(u_{+})$, with the analogous estimates obtained now from $\dot{s}_{-}$ instead; in this case, $\sigma-\dot{s}_{-}(t)$ is uniformly bounded above by some negative number and $\dot{s}_+-\sigma\leq0$, which implies again that $\dot{s}_+-\dot{s}_-$ is uniformly negative. Since $s_{+}(0)-s_{-}(0)=x_{+}-x_{-}$, this means that the two shocks must meet in finite time. Thus, we are done proving Lemma \ref{main lemma} in this special case, which also crucially proves Theorem \ref{SingleShock} for initial data such that $u_{+}\leq u_0\leq u_{-}$.

Hence, to complete the proof, it is enough to show that in the general case, the $s_{+}$ shock curve meets the $u_{-}$ characteristic (call it $y_{-}$) in the rarefaction emanating from $x_0$, and that the $s_{-}$ shock curve meets the $u_{+}$ characteristic (call it $y_{+}$) emanating from the same rarefaction. This reduces us to the special case, which we already proved. Again, we only consider the case of $s_{+}$ and $y_{-}$, and the other case will follow mutatis mutandis by a similar argument.

Note that if $x_0=x_{+}$, there is no shock curve emanating from $x_{+}$. However, there is still a rarefaction fan with maximal forward characteristic $y_{+}$ from the point $x_0$, along which $u(y_+(t),t)\equiv u_{+}$. Thus, once $y_{+}$ and $s_{-}$ meet, we immediately see a single shock emerge from the point of their interaction. 

Similarly, if $x_0=x_-$, the rarefaction fan at $x_0$ has minimal forward characteristic $y_{-}$ with $u(y_-(t),t)\equiv u_{-}$, and once again we see a single shock emerge after $y_{-}$ and $s_+$ meet. Hence, the following proof covers all cases (once it is adapted, mutatis mutandis, for $s_-$ and $y_+$).

Given some $T>0$ such that $s_{+},y_{-}$ do not meet before $t=T$, consider the region $\Gamma_T$ bounded by the curves $y_{-}(t),s_{+}(t)$ on $(0,T)$, and the time levels $t=0,t=T$. Let $u$ denote the entropy solution of the Cauchy problem with piecewise constant data \eqref{special data}. Note that $u_r(y_{-}(t),t)=u_{-}$ for a.e. $t\in(0,T)$. Since $u_{+}$ is a stationary solution of \eqref{claw}, we have that
\[
(u-u_{+})_t+(f(x,u)-f(x,u_{+}))_x=0.
\]
Hence, by Dafermos' `divergence theorem' for scalar conservation laws \cite{GenChar}:
\[
\begin{split}
	0=&\int_{y_{-}(T)}^{s_{+}(T)}u(x,T)-u_{+}dx-\int_{x_0}^{x^{+}}u_M-u_{+}dx \\
	&+\int_{0}^{T}f(s_{+}(t),u_l(s_{+}(t),t))-f(u_{+})-\dot{s}_{+}(t)(u_l(s_{+}(t),t)-u_{+})dt \\
	&-\int_{0}^{T}f(u_{-})-f(u_{+})-\dot{y}_-(t)(u_{-}-u_{+})dt \\
	0=&\int_{y_{-}(T)}^{s_{+}(T)}u(x,T)-u_{+}dx-\int_{x_0}^{x^{+}}u_M-u_{+}dx \\
	&+\int_{0}^{T}f(s_{+}(t),u_l(s_{+}(t),t))-f(u_{+})dt \\
	&-\int_{0}^{T}\dfrac{f(s_{+}(t),u_l(s_{+}(t),t))-f(u_{+})}{u_l(s_{+}(t),t)-u_{+}}(u_l(s_{+}(t),t)-u_{+})dt \\
	&+\int_{0}^{T}f(u_{+})-f(u_{-})-f_{u}(y_{-}(t),u_{-})(u_{+}-u_{-})dt,
\end{split}
\]
but the integral along $s_{+}$ has integrand zero, and therefore
\[
\begin{split}
	0=&\int_{y_{-}(T)}^{s_{+}(T)}u(x,T)-u_{+}dx-\int_{x_0}^{x^{+}}u_M-u_{+}dx \\
	&+\int_{0}^{T}f(u_{+})-f(u_{-})-f_{u}(y_{-}(t),u_{-})(u_{+}-u_{-})dt.
\end{split}
\]
Hence by \ref{UC}, we obtain the following inequality for all $T$ such that $y_{-},s_{+}$ do not meet before $t=T$.
\[
0\leq\int_{y_{-}(T)}^{s_{+}(T)}u(x,T)-u_{+}dx\leq(u_M-u_{+})(x_{+}-x_0)-\frac{\alpha}{2}\abs{u_{+}-u_{-}}^2T.
\]
Suppose, if possible, that the curves $s_{+},y_{-}$ never meet, so that this inequality is valid for all $T>0$. However, for large enough $T$, the right-hand side term will be negative, which contradicts our assumption that the curves never meet. Note that if $x_{+}=x_0$, then there is nothing to prove. Hence, $y_{-}$ must meet $s_{+}$ in finite time. Similarly, $y_{+}$ must meet $s_{-}$ in finite time, and there is nothing to prove in this case if $x_{-}=x_0$. This concludes the proof.

\subsection{Counterexample with non-stationary solution}
For $v\in\mathbb{R}$ such that $f_x(\cdot,v)\neq0$, it is trivial to see that the constant function $v$ is not a solution of \eqref{claw}. Without assuming \eqref{S}, however, we cannot expect exactly one single shock to emerge in finite time as Theorem~\ref{SingleShock} states, under the same assumptions on the initial data. This is because constant initial data may give rise to shocks in positive time, not just non-trivial solutions. However, it would be interesting to consider more general cases of shock-free entropy solutions, e.g. smooth functions $u(x,t)=g(x)$ such that $f(x,g(x))=0$.

Let us illustrate this with an explicit example; let $u_0\equiv1/2$ be constant initial data, for the Cauchy problem \eqref{claw} with flux given by
\[
f(x,u)=\left(1+\exp\left(-x^2\right)\right)(u^2-u),
\] 
an instance of a heterogeneous LWR-type flux mentioned earlier. The characteristic equations \eqref{char ode} in this case are given by
\begin{equation}\label{char ode special}
	\begin{split}
		\dot{y}(t)&=\left(1+\exp\left(-y(t)^2\right)\right)(2z(t)-1), \\
		\dot{z}(t)&=2y(t)\exp\left(-y(t)^2\right)(z(t)^2-z(t)),
	\end{split}
\end{equation}
which must be solved for initial conditions $y(0)=x,z(0)=1/2$ for each $x\in\mathbb{R}$. Since the initial data is constant, a Lipschitz solution exists at least locally in time, and thus forward characteristics emanating from $t=0$ are unique for each $x$. By explicit computation, we can see that $(y(t),z(t))=(0,1/2)$ is a solution of \eqref{char ode special}. However, for $x>0$, we can explicitly calculate that
\[
\dot{z}(0)=2x\exp\left(-x^2\right)\dfrac{-1}{4},
\]
and hence at a small enough positive time $\epsilon$, we must have that $z(\epsilon)<1/2$ and $y(\epsilon)>0$. We claim that $y$ reaches zero in finite time. Note first that $z(t)\equiv0,1$ are solutions of the characteristic equations, since $u\equiv0,1$ are classical stationary solutions of \eqref{claw} with the given flux. Hence, by Cauchy-Lipschitz theory for ODEs, we have that $z(t)\in(0,1)$ for all $t$. Hence, as long as $y>0$, by \eqref{char ode special} we have that
\[
\dot{z}(t)=2y(t)\exp\left(-y(t)^2\right)(z(t)^2-z(t))<0,
\]
so that $z(t)\leq z(\epsilon)<1/2$ as long as $y(t)>0$. Since $1\leq1+\exp(-y^2)\leq2$ and $2z(t)-1\leq2z(\epsilon)-1<0$, we obtain the uniform bound
\[
\dot{y}(t)=\left(1+\exp\left(-y(t)^2\right)\right)(2z(t)-1)\leq2z(\epsilon)-1<0\quad\text{ as long as }y(t)>0.
\]
Therefore, $y(t^{\ast})=0$ at some finite time $t^{\ast}\leq\epsilon+y(\epsilon)/(1-2z(\epsilon))$, and it carries the value $z(t^*)\neq1/2$, forcing shock formation.

Hence, any characteristic emanating from $x>0,t=0$ cannot fail to meet the characteristic emanating from $x=t=0$ in finite time. A similar computation can be carried out, mutatis mutandis, for characteristics that emanate from $x<0,t=0$. In particular, the entropy solution cannot remain shock-free for all time, even though the initial data is as simple as possible.

This example can easily be augmented to demonstrate solutions exhibiting infinitely many shocks from constant initial data by \ref{FSP}. Let $\varphi\in C_c^{\infty}(\mathbb{R})$ be non-negative, even, supported in $(-1,1)$ such that $\varphi(x)=\exp(-x^2)$ for $x\in[-1/2,1/2]$ and is monotone decreasing in $[0,1]$. Now, for a spacing $D>2$ to be chosen later, consider the flux
\[
f(x,u)=\left(1+\sum_{k\in\mathbb{Z}}\varphi(x-Dk)\right)(u^2-u),
\]
which is smooth and satisfies \eqref{FSP}. Again, $u\equiv0,1$ are the stationary solutions of \eqref{claw} with this flux, and we can consider $u_0\equiv1/2$ to force shock formation on each interval of the form $(Dk-1,Dk+1)$ by translation invariance. Choosing $D$ large enough such that the intervals of non-homogeneity are sufficiently distant from one another that they behave independently, we obtain that the entropy solution of the Cauchy problem has infinitely many shocks.

\section{Proofs II}\label{proofs2}
We employ several lemmas to prove the theorem, after which the corollary follows quite easily. Consider \eqref{Het entropy}, the entropy equation that admissible solutions of \eqref{claw} satisfy in the sense of distributions. Since this equation is not in strict divergence form, we consider a family of entropies for which admissible solutions satisfy such an equation. In particular, we restrict ourselves to quadratic entropies.

\subsection{Stability of the simple shock}\label{L2 proof}
After the following lemma, the rest of our argument follows along the lines of Leger's proof in \cite{LegerShock} for the homogeneous case, with some technical changes to account for heterogeneous terms. While the core ideas are similar, we include full proofs of most lemmas for the sake of completeness. More interestingly, \ref{P} is a sharp assumption on the heterogeneity with respect to this generalisation; in section~\ref{neg het} we demonstrate that this result is false if the heterogeneity violates our assumption \ref{P}.
\begin{lemma}
	Suppose the flux $f$ satisfies the assumptions \ref{S}-\ref{FSP}{, as well as \eqref{P}.} For each $c \in [u^+, u^-]$, define the convex entropy $\eta(\cdot, c)$ by:
	\[
	\eta(u,c)=(u-c)^2.
	\]
	Define the associated fluxes $Q(x,u,c)$ as
	\[
	Q(x,u,c)=\int_{c}^{u}2(y-c)f_{u}(x,y)dy.
	\]
	{Then, if $u$ satisfies \eqref{Het entropy}, it also satisfies}
	\[
	\eta(u,c)_t+Q(x,u,c)_x\leq0.
	\]
\end{lemma}
\begin{proof}
	By \eqref{Het entropy}, it is enough for us to prove that
	\[
	\eta_u(u,c)f_x(x,u)-Q_x(x,u,c)\geq0.
	\]
	Note that both terms on the left-hand side are essentially bounded functions, so it is enough to prove the inequality pointwise in $x$. Furthermore, we suggestively denote the partial derivative of $\eta$ with respect to $u$ by $\eta^{\prime}$, thus $2(u-c)=\eta^{\prime}(u,c)$. By differentiating under the integral sign, we can write $Q_x$ as
	\[
	\begin{split}
		-Q_x(x,u,c)&=-\pdv{}{x}\int_{c}^{u}2(y-c)f_u(x,y)dy \\
		&=-\int_{c}^{u}\eta^{\prime}(y,c)f_{xu}(x,y)dy \\
		&=\int_{c}^{u}\eta^{\prime\prime}(y,c)f_x(x,y)dy-\eta^{\prime}(u,c)f_x(x,u) \\
		&=-\eta^{\prime}(u,c)f_x(x,u)+2\int_{c}^{u}f_x(x,y)dy.
	\end{split}
	\]
	Hence, we have that
	\[
	\eta^{\prime}(u,c)f_x(x,u)-Q_x(x,u,c)=2\int_{c}^{u}f_x(x,y)dy\geq0,
	\]
	where the last inequality holds because \ref{P} implies that $f_x$ is an increasing function of its second argument and vanishes at $c$. Therefore, we have that
	\[
	\eta^{\prime}(u,c)f_x(x,u)-Q_x(x,u,c)\geq0,
	\]
	which completes the proof.
\end{proof}

Now, from the entropy inequality, and Dafermos' `divergence theorem' for scalar conservation laws \cite{GenChar}, the lemma below, which we present without proof, trivially follows. Recall that by the correspondence with Hamilton-Jacobi equations, left and right traces in $x$ are well defined everywhere for entropy solutions of \eqref{claw}.
\begin{lemma}\label{divergence theorem}
	Let $0\leq a<b;\xi_1,\xi_2:[a,b]\to\mathbb{R}$ such that $\xi_i$ are Lipschitz and $\xi_1-\xi_2\leq -\delta<0$. Then, for each fixed $c\in[u_{+},u_{-}]$, we have that
	\[
	\begin{split}
		&\int_{\xi_1(b)}^{\xi_2(b)}\eta(u(x,b),c)dx-\int_{\xi_1(a)}^{\xi_2(a)}\eta(u(x,a),c)dx \\
		\leq&\int_{a}^{b}\eta(u(\xi_2(t)-,t),c)\dot{\xi}_2(t)-Q(\xi_2(t),u(\xi_2(t)-,t),c)dt \\
		&-\int_{a}^{b}\eta(u(\xi_1(t)+,t),c)\dot{\xi}_1(t)-Q(\xi_1(t),u(\xi_1(t)+,t),c)dt.
	\end{split}
	\]
\end{lemma}
Furthermore, when the initial data is an integrable perturbation of the special Riemann data, we have the following lemma. Since the proof is a trivial adaptation of the argument given by Leger in \cite{LegerShock}, we omit it.
\begin{lemma}\label{Leger overkill}
	Let $u$ be the solution to the Cauchy problem \eqref{claw} with initial data $u_0\in L^{\infty}(\mathbb{R})$ such that $\pnorm{2}{u_0-\Phi}<\infty$, where $\Phi$ is the special Riemann data as defined as in \eqref{Riemann}. Then, for any Lipschitz curve $\xi:[a,b]\to\mathbb{R}$, we have the following two inequalities:
	\[
	\begin{split}
		&\int_{-\infty}^{\xi(b)}\eta(u(x,b),c)dx-\int_{-\infty}^{\xi(a)}\eta(u(x,a),c)dx \\
		\leq&\int_{a}^{b}\eta(u(\xi(t)-,t),c)\dot{\xi}(t)-Q(\xi(t),u(\xi(t)-,t),c)dt,
	\end{split}
	\]
	and
	\[
	\begin{split}
		&\int_{\xi(b)}^{\infty}\eta(u(x,b),c)dx-\int_{\xi(a)}^{\infty}\eta(u(x,a),c)dx \\
		\leq&-\int_{a}^{b}\eta(u(\xi(t)+,t),c)\dot{\xi}(t)-Q(\xi(t),u(\xi(t)+,t),c)dt.
	\end{split}
	\]
\end{lemma}
Now, suppose we have two Lipschitz curves $\xi_{\pm}:[0,\infty)\to\mathbb{R}$ such that $\xi_{\pm}(0)=0$ and $\xi_{+}\leq\xi_{-}$ for all time. Then, by Lemma~\ref{Leger overkill} we obtain following inequality: for any curve $\tilde{\xi}$ such that $\xi_{+}\leq\tilde{\xi}\leq\xi_{-}$ for all time, we have that for all $T>0$:
\[
\begin{split}
	\pnorm{2}{u(\cdot,T)-\Phi(\cdot-\tilde{\xi}(T))}^2=&\int_{-\infty}^{\tilde{\xi}(T)}\eta(u(x,T),u_{-})dx+\int_{\tilde{\xi}(T)}^{\infty}\eta(u(x,T),u_{+})dx \\
	\leq&\int_{-\infty}^{\xi_{-}(T)}\eta(u(x,T),u_{-})dx+\int_{\xi_{+}(T)}^{\infty}\eta(u(x,T),u_{+})dx \\
	\leq&\int_{-\infty}^0\eta(u_0(x),u_{-})dx+\int_{0}^{\infty}\eta(u_0(x),u_{+})dx \\
	&+\int_{0}^{T}\eta(u(\xi_{-}(t)-,t),u_{-})\dot{\xi}_{-}(t)-Q(\xi_{-}(t),u(\xi_{-}(t)-,t),u_{-})dt \\
	&-\int_{0}^{T}\eta(u(\xi_{+}(t)+,t),u_+)\dot{\xi}_{+}(t)-Q(\xi_{+}(t),u(\xi_{+}(t)+,t),u_{+})dt,
\end{split}
\]
and hence
\[
\begin{split}
	\pnorm{2}{u(\cdot,T)-\Phi(\cdot-\tilde{\xi}(T))}^2\leq&\pnorm{2}{u_0-\Phi}^2 \\
	&+\int_{0}^{T}\eta(u(\xi_{-}(t)-,t),u_{-})\dot{\xi}_{-}(t)-Q(\xi_{-}(t),u(\xi_{-}(t)-,t),u_{-})dt \\
	&-\int_{0}^{T}\eta(u(\xi_{+}(t)+,t),u_+)\dot{\xi}_{+}(t)-Q(\xi_{+}(t),u(\xi_{+}(t)+,t),u_{+})dt.
\end{split}
\]
Thus, it is enough for us to find Lipschitz functions $\xi_{+}(t),\xi_{-}(t)$, that respectively satisfy the inequalities
\[
\eta(u(\xi_{+}(t)+,t),u_+)\dot{\xi}_{+}(t)-Q(\xi_{+}(t),u(\xi_{+}(t)+,t),u_{+})\geq0,
\]
and
\[
\eta(u(\xi_{-}(t)-,t),u_{-})\dot{\xi}_{-}(t)-Q(\xi_{-}(t),u(\xi_{-}(t)-,t),u_{-})\leq0,
\]
such that $\xi_{\pm}(0)=0$ and $\xi_{+}\leq\xi_{-}$ for all time. To construct these curves, let us define the `normalised' entropy flux $q(x,u,c)$ as follows for $c\in[u_{+},u_{-}],u\neq c$:
\begin{equation}\label{monotono q}
	q(x,u,c)=\dfrac{Q(x,u,c)}{\eta(u,c)}=\int_{c}^{u}\dfrac{\eta^{\prime}(y,c)}{\eta(u,c)}f_u(x,y)dy=\int_{0}^{1}2zf_u(x,c+z(u-c))dz,
\end{equation}
where the second equality follows from the change of variables $z=(y-c)/(u-c)$. In this latter version, it is clear that $q$ is continuously differentiable in all variables. Furthermore, by \ref{UC}, we have that for all $x,u\in\mathbb{R},c\in[u_{+},u_{-}]$:
\begin{equation}\label{c growth}
	\begin{split}
		\pdv{q}{u}(x,u,c)&\geq\dfrac{2\alpha}{3}, \\
		\pdv{q}{c}(x,u,c)&\geq\dfrac{\alpha}{3}.
	\end{split}
\end{equation}
Hence, to complete the proof, it is sufficient to find Lipschitz curves $\xi_{\pm}$ with $\xi_{\pm}(0)=0,\xi_{+}-\xi_{-}\leq0$ such that for almost all $t$, we have that
\begin{equation}\label{pm curves}
	\begin{split}
		\dot{\xi}_{-}(t)&\leq q(\xi_{-}(t),u(\xi_{-}(t)-,t),u_{-}), \\
		\dot{\xi}_{+}(t)&\geq q(\xi_{+}(t),u(\xi_{+}(t)+,t),u_{+}).
	\end{split}
\end{equation}
Thus, motivated by Dafermos' generalised characteristics, which emerge as solutions to ODEs in the sense of Filippov \cite{Filippov}, we consider Lipschitz solutions of the differential inclusion
\begin{equation}\label{diffclusion}
	\dot{\xi}(t)\in[q(\xi(t),u(\xi(t)+,t),c),q(\xi(t),u(\xi(t)-,t),c)].
\end{equation}
Note that this inclusion is well defined due to \eqref{c growth}, since $u$ is assumed to be an entropy solution, satisfying $u(x+,t)<u(x-,t)$ at points of spatial discontinuity. Thus, given that $q$ is strictly increasing in $u$, the existence and uniqueness (for positive times) of solutions to these differential inclusions follows from {\cite[Proposition 3, appendix]{LegerShock}}. Furthermore, a lemma from Dafermos {\cite[Lemma 3.1]{GenChar}} restricts their behaviour at points of discontinuity; in particular, solutions of \eqref{diffclusion} must satisfy the Rankine-Hugoniot conditions along discontinuities.
\begin{lemma}
	Let $\xi:[t_0,T]\to\mathbb{R}$ be a Lipschitz solution to \eqref{diffclusion} for some $t_0\geq0$. Then, for almost all $t\in[t_0,T]$, we have that
	\[
	\dot{\xi}(t)=
	\begin{cases}
		&q(\xi(t),u(\xi(t),t),c)\text{ if }u(\xi(t)-,t)=u(\xi(t)+,t), \\
		\smallskip
		&\dfrac{f(\xi(t),u(\xi(t)-,t))-f(\xi(t),u(\xi(t)+,t))}{u(\xi(t)-,t)-u(\xi(t)+,t)}\text{ if }u(\xi(t)-,t)>u(\xi(t)+,t).
	\end{cases}
	\]
\end{lemma}
\begin{proof}
    For a.e. $t\in[t_0,T]$, the derivative $\dot{\xi}$ exists. If $u(\xi(t)-,t)=u(\xi(t)+,t)$, then the interval in \eqref{diffclusion} trivially reduces to a single point, which is precisely the first case. For the second case, let $a,b\in\mathbb{R}$ such that for all $t\in[t_0,T]$ we have that $a<\xi(t)<b$. Now for any times $t_0\leq\tau_1<\tau_2\leq T$, apply the divergence theorem to three regions:
    \begin{enumerate}
        \item The rectangular domain $[a,b]\times[\tau_1,\tau_2]$.
        \item The left portion $\left\{(x,t):\tau_1\leq t\leq \tau_2,a\leq x\leq\xi(t)\right\}$.
        \item The right portion $\left\{(x,t):\tau_1\leq t\leq \tau_2,\xi(t)\leq x\leq b\right\}$.
    \end{enumerate}
    Since the integrals over the boundaries of the rectangular domain sum to zero by Dafermos' `divergence theorem' for scalar conservation laws \cite{GenChar}, we conclude that
    \[
	\int_{\tau_1}^{\tau_2}\dot{\xi}(t)\left[u(\xi(t)-,t)-u(\xi(t)+,t)\right]-\left[f(\xi(t),u(\xi(t)-,t))-f(\xi(t),u(\xi(t)+,t))\right]dt=0.
	\]
    Since $\tau_1<\tau_2$ were arbitrary times in the domain $[t_0,T]$, it follows that the integrand vanishes for almost all $t\in[t_0,T]$. Hence for almost all $t\in[t_0,T]$ such that $u(\xi(t)-,t)>u(\xi(t)+,t)$, the Rankine-Hugoniot condition must hold.
\end{proof}
Thus, let $\xi_{+},\xi_{-}$ respectively, with $\xi_{\pm}(0)=0$, be solutions of the following differential inclusions:
\begin{equation}\label{final step}
	\begin{split}
		\dot{\xi}_{+}(t)&\in[q(\xi_{+}(t),u(\xi_{+}(t)+,t),u_+),q(\xi_{+}(t),u(\xi_{+}(t)-,t),u_+)], \\
		\dot{\xi}_{-}(t)&\in[q(\xi_{-}(t),u(\xi_{-}(t)+,t),u_-),q(\xi_{-}(t),u(\xi_{-}(t)-,t),u_-)].
	\end{split}
\end{equation}
These curves certainly satisfy the inequalities \eqref{pm curves}, and thus the proof will be complete if we can show that $\xi_{+}-\xi_{-}\leq0$ for all time. In order to do so, we will employ the following lemmas and prove the required result by contradiction.
\begin{lemma}\label{lambda}
	Let $\xi_{\pm}$ be solutions to the differential inclusion \eqref{diffclusion} on the interval $[t_0,t_1]\subset[0,\infty)$ such that $\forall t\in[t_0,t_1]:\xi_{+}(t)-\xi_{-}(t)\geq\delta>0$. Let $S_{a,b}=\{t\in[t_0,t_1]:\dot{\xi}_{+}(t)\geq\dot{\xi}_{-}(t)\}$. Then, for $t_0\leq a<b\leq t_1$ and $c\in(u_{+},u_{-})$, there exists $\lambda>0$ independent of $\delta$ such that
	\[\int_{\xi_{-}(b)}^{\xi_{+}(b)}\eta(u(x,b),c)dx-\int_{\xi_{-}(a)}^{\xi_{+}(a)}\eta(u(x,a),c)dx\leq-\lambda\abs{S_{a,b}}.\]
\end{lemma}
\begin{proof}
	From Lemma~\ref{divergence theorem}, we have that
	\[
	\begin{split}
		&\int_{\xi_{-}(b)}^{\xi_{+}(b)}\eta(u(x,b),c)dx-\int_{\xi_{-}(a)}^{\xi_{+}(a)}\eta(u(x,a),c)dx \\
		\leq&-\int_{a}^{b}\eta(u(\xi_{-}(t)+,t),c)\dot{\xi}_{-}(t)-Q(\xi_{-}(t),u(\xi_{-}(t)+,t),c)dt \\
		&+\int_{a}^{b}\eta(u(\xi_{+}(t)-,t),c)\dot{\xi}_{+}(t)-Q(\xi_{+}(t),u(\xi_{+}(t)-,t),c)dt \\
		=&-\int_{a}^{b}\eta(u(\xi_{-}(t)+,t),c)[\dot{\xi}_{-}(t)-q(\xi_{-}(t),u(\xi_{-}(t)+,t),c)]dt \\
		&+\int_{a}^{b}\eta(u(\xi_{+}(t)-,t),c)[\dot{\xi}_{+}(t)-q(\xi_{+}(t),u(\xi_{+}(t)-,t),c)]dt.
	\end{split}
	\]
	Since $q$ is strictly monotone in $c$, we have that
	\[
	\dot{\xi}_{-}(t)-q(\xi_{-}(t),u(\xi_{-}(t)+,t),c)\geq\dfrac{\alpha}{3}(u_{-}-c),
	\]
	and therefore
	\begin{equation}\label{mono}
		\begin{split}
			-[\dot{\xi}_{-}(t)-q(\xi_{-}(t),u(\xi_{-}(t)+,t),c)]&\leq\dfrac{\alpha}{3}(c-u_{-})<0, \\
			\dot{\xi}_{+}(t)-q(\xi_{+}(t),u(\xi_{+}(t)-,t),c)&\leq\dfrac{\alpha}{3}(u_{+}-c)<0.
		\end{split}
	\end{equation}
	Note that $\eta\geq0$, hence the signs are well-behaved. However, $\eta$ can be zero or arbitrarily close to zero. Hence, to prove the lemma, we concentrate on the set $S_{a,b}\subseteq[a,b]$ as defined above. By \eqref{mono} and the non-negativity of $\eta$, we have that
	\begin{equation}\label{Sabo}
		\begin{split}
			&\int_{\xi_{-}(b)}^{\xi_{+}(b)}\eta(u(x,b),c)dx-\int_{\xi_{-}(a)}^{\xi_{+}(a)}\eta(u(x,a),c)dx \\
			\leq&-\int_{S_{a,b}}\eta(u(\xi_{-}(t)+,t),c)[\dot{\xi}_{-}(t)-q(\xi_{-}(t),u(\xi_{-}(t)+,t),c)]dt \\
			&+\int_{S_{a,b}}\eta(u(\xi_{+}(t)-,t),c)[\dot{\xi}_{+}(t)-q(\xi_{+}(t),u(\xi_{+}(t)-,t),c)]dt.
		\end{split}
	\end{equation}
	For $t\in S_{a,b}$ and by \eqref{final step} we have the following inequalities:
	\[
	q(\xi_{+}(t),u(\xi_{+}(t)-,t),u_{+})\geq\dot{\xi}_{+}(t)\geq\dot{\xi}_{-}(t)\geq q(\xi_{-}(t),u(\xi_{-}(t)+,t),u_{-}).
	\]
	Now note that by the mean value theorem and \ref{UC},
	\[
	q(\xi_{+}(t),u(\xi_{+}(t)-,t),u_{-})-q(\xi_{+}(t),u(\xi_{+}(t)-,t),u_{+})\geq\dfrac{\alpha}{3}(u_{-}-u_{+}).
	\]
    We now exploit the homogeneity of $f$ enforced by our assumption~\ref{P}, as recorded in Remark~\ref{band remark}:
	\begin{equation}\label{band flat}
		f_{x}(x,u)=0\quad\text{for all }x\in\mathbb{R},\ u\in[u_{+},u_{-}],
	\end{equation}
    since for each fixed $x$ the function $u\mapsto f_{x}(x,u)$ is non-decreasing by \ref{P} and vanishes at $u=u_{\pm}$ by \ref{S}. Hence, differentiating \eqref{band flat} with respect to $u$ and by \eqref{C2}, we have that $f_{xu}(\cdot,u)=0$ whenever $u\in[u_{+},u_{-}]$. Now by \ref{P} again, we have that for $u\in[u_{+},u_{-}]$:
	\[
	\pdv{q}{x}(x,u,c)=\int_{0}^{1}2zf_{xu}(x,c+z(u-c))dz\geq0.
	\]
    Hence, we conclude that $q(x,u,c)$ is independent of $x$ whenever $u,c\in[u_{+},u_{-}]$. Thus, the quantity
	\begin{equation}\label{M}
	    M=\sup\left\{\pdv{q}{u}(x,u,c):x\in\mathbb{R},\ u,c\in[u_{+},u_{-}]\right\}
	\end{equation}
	is the supremum of a continuous function of $(u,c)$ alone over the compact set $[u_{+},u_{-}]^{2}$, hence finite. Note that $M$ depends only on $f$ and $u_{\pm}$, and that $M\geq2\alpha/3>0$ by \eqref{c growth}. In order to obtain a uniform rate of dissipation for the quantity of interest, we claim that for every $c\in(u_{+},u_{-})$ and $t\in S_{a,b}$:
	\begin{equation}\label{eta lower}
		\max\{\eta(u(\xi_{-}(t)+,t),c),\eta(u(\xi_{+}(t)-,t),c)\}\geq\beta(c),
	\end{equation}
    where
    \[
    \beta(c)=\min\left\{(u_{-}-c)^{2},(c-u_{+})^{2},\dfrac{\alpha^{2}}{36M^{2}}(u_{-}-u_{+})^{2}\right\}>0.
    \]
    To see why this is true, consider its disjunction into three exhaustive cases. For simplicity of notation, we denote
    \[
	v=u(\xi_{-}(t)+,t),\qquad w=u(\xi_{+}(t)-,t)
	\]
    for $t\in S_{a,b}$. Since our claim is that bound is independent of time (conditional on $t\in S_{a,b}$), dropping $t$ from the notation does not lead to any loss of generality since our argument holds for any arbitrary $t\in S_{a,b}$.

    \textit{Case 1: $u(\xi_{-}(t)+,t)\notin[u_{+},u_{-}]$.} Then either $u(\xi_{-}(t)+,t)>u_{-}>c$ or $u(\xi_{-}(t)+,t)<u_{+}<c$, so that
	\[
	\eta(u(\xi_{-}(t)+,t),c)=(u(\xi_{-}(t)+,t)-c)^{2}\geq\min\{(u_{-}-c)^{2},(c-u_{+})^{2}\}\geq\beta(c).
	\]

	\textit{Case 2: $u(\xi_{+}(t)-,t)\notin[u_{+},u_{-}]$.} Exactly as in Case 1, we have that $\eta(u(\xi_{+}(t)-,t),c)\geq\min\{(u_{-}-c)^{2},(c-u_{+})^{2}\}\geq\beta(c)$.

    \textit{Case 3: $u(\xi_{-}(t)+,t),u(\xi_{+}(t)-,t)\in[u_{+},u_{-}]$.} Since $u(\xi_{-}(t)+,t),u_{-}\in[u_{+},u_{-}]$, we have that:
	\[
	q(\xi_{-}(t),u(\xi_{-}(t)+,t),u_{-})=q(\xi_{+}(t),u(\xi_{-}(t)+,t),u_{-}).
	\]
	Hence, we have that:
	\begin{equation}
		\begin{split}
			0<\dfrac{\alpha}{3}(u_{-}-u_{+})&\leq q(\xi_{+}(t),u(\xi_{+}(t)-,t),u_{-})-q(\xi_{-}(t),u(\xi_{-}(t)+,t),u_{-}) \\
			&=q(\xi_{-}(t),u(\xi_{+}(t)-,t),u_{-})-q(\xi_{-}(t),u(\xi_{-}(t)+,t),u_{-}) \\
			&=\pdv{q}{u}(\xi_{-}(t),\tilde{u},u_{-})(u(\xi_{+}(t)-,t)-u(\xi_{-}(t)+,t)) \\
			&\leq M(u(\xi_{+}(t)-,t)-u(\xi_{-}(t)+,t)),
		\end{split}
	\end{equation}
	where $\tilde{u}$ lies in $(u(\xi_{+}(t)-,t),u(\xi_{-}(t)+,t))$ and $M$ is defined as in~\eqref{M}. Note that $q_u$ is also a priori known to be positive, so the inequality is valid. Hence, for $t\in S_{a,b}$, we have the explicit lower bound
	\[
	u(\xi_{+}(t)-,t)-u(\xi_{-}(t)+,t)\geq\dfrac{\alpha}{3M}(u_{-}-u_{+}).
	\]
	Thus, for any $c\in(u_{+},u_{-})$ and $t\in S_{a,b}$, we must have that either
	\[
	\eta(u(\xi_{-}(t)+,t),c)\geq\dfrac{\alpha^2}{36M^2}(u_{-}-u_{+})^2\geq\beta(c),
	\]
	or
	\[
	\eta(u(\xi_{+}(t)-,t),c)\geq\dfrac{\alpha^2}{36M^2}(u_{-}-u_{+})^2\geq\beta(c).
	\]
    Finally, both terms in the integrand on the right-hand side of \eqref{Sabo} are non-positive by \eqref{mono} and the non-negativity of $\eta$. Hence, retaining for each $t\in S_{a,b}$ only the term corresponding to the larger of $\eta(v,c),\eta(w,c)$, we get from \eqref{mono} and \eqref{eta lower} that
	\[
	\begin{split}
		&\int_{\xi_{-}(b)}^{\xi_{+}(b)}\eta(u(x,b),c)dx-\int_{\xi_{-}(a)}^{\xi_{+}(a)}\eta(u(x,a),c)dx \\
		\leq&-\dfrac{\alpha}{3}\min{\{u_{-}-c,c-u_{+}\}}\beta(c)\cdot\abs{S_{a,b}},
	\end{split}
	\]
	which proves the lemma with
    \[
    \lambda=\frac{\alpha}{3}\min{\{u_{-}-c,c-u_{+}\}}\beta(c)>0,
    \]
    as required.
\end{proof}
\begin{lemma}\label{shrinkage}
	Let $\xi_{\pm}$ be solutions to the differential inclusion(s) \eqref{diffclusion} on the interval $[t_0,t_1]\subset[0,\infty)$ such that $\forall t\in[t_0,t_1]:\xi_{+}(t)-\xi_{-}(t)\geq\delta>0$, with $\xi_{+}(t_0)-\xi_{-}(t_0)=\delta$. Then, there exists $\kappa\geq1$, independent of $\delta$, such that for all $t\in[t_0,t_1]$:
	\[
	\xi_{+}(t)-\xi_{-}(t)\leq\kappa\delta.
	\]
\end{lemma}
\begin{proof}
	Let $\theta(t)=\xi_{+}(t)-\xi_{-}(t)$. Note that the curves $\xi_{\pm}$ solve the differential inclusion \eqref{diffclusion} and are thus uniformly Lipschitz with $\abs{\dot{\xi}_{\pm}}\leq\sup q(x,\tilde{u},c)$, where the supremum is taken over $x\in\mathbb{R},\abs{\tilde{u}}\leq\supnorm{u},c\in(u_{+},u_{-})$; note that
	\[
	L=\sup\left\{\abs{f_u(x,v)}:x\in\mathbb{R},\abs{v}\leq\max\left\{\supnorm{u},\abs{u_{+}},\abs{u_{-}}\right\}\right\}
	\]
	is finite by~\ref{FSP}; further note that by \eqref{monotono q} this functions as a bound on $q$ as well, and $\abs{\dot{\theta}}\leq2L$. By the fundamental theorem of calculus, we have that for $t\in[t_0,t_1]$:
	\begin{equation}\label{semicontraction step}
		\begin{split}
			\xi_{+}(t)-\xi_{-}(t)&=\delta+\int_{t_0}^{t}\dot{\theta}(s)ds \\
			&\leq\delta+\int_{S_{t_0,t}}\dot{\theta}(s)ds \\
			&\leq\delta+2L\abs{S_{t_0,t}}.
		\end{split}
	\end{equation}
	Now, let $N$ be an upper bound for $\eta$ over the compact set $[-\supnorm{u},\supnorm{u}]\times[u_{+},u_{-}]$.
	By Lemma \ref{lambda}, we have that for $c\in(u_{+},u_{-})$:
	\[
	\begin{split}
		0&\leq\int_{\xi_{-}(t)}^{\xi_{+}(t)}\eta(u(x,t),c)dx \\
		&\leq-\lambda\abs{S_{t_0,t}}+\int_{\xi_{-}(t_0)}^{\xi_{+}(t_0)}\eta(u(x,t_0),c)dx \\
		&\leq-\lambda\abs{S_{t_0,t}}+N\delta,
	\end{split}
	\]
	and therefore,
	\[
	\abs{S_{t_0,t}}\leq\dfrac{N\delta}{\lambda}.
	\]
	Hence, we can rewrite the inequality \eqref{semicontraction step} as
	\[
	\xi_{+}(t)-\xi_{-}(t)\leq\left(1+\dfrac{2LN}{\lambda}\right)\delta,
	\]
	which completes the proof.
\end{proof}
Finally, let us state and prove the lemma that will complete the proof of our theorem.
\begin{lemma}
	Let $\xi_{\pm}$ be solutions to the differential inclusion(s) \eqref{final step} with $\xi_{\pm}(0)=0$. Then, $\forall t\geq0:\xi_{+}(t)\leq\xi_{-}(t)$.
\end{lemma}
\begin{proof}
	We will employ proof by contradiction. Suppose $\xi_{+}(T)-\xi_{-}(T)=G>0$, and define $d(t)=\xi_{+}(t)-\xi_{-}(t)$ for $t\in[0,T]$, so that $d$ is Lipschitz continuous and $d(0)=0,d(T)=G>0$. Let $0<\delta<G$ and consider the set $d^{-1}(\delta)\subset[0,T]$; by the intermediate value theorem this set is non-empty, by the continuity of $d$ it is closed, and therefore $t_{\delta}=\sup d^{-1}(\delta)$ is well defined and $t_{\delta}\in d^{-1}(\delta)$. Now, $d(t)>\delta$ for $t\in(t_{\delta},T]$ and $d(t_{\delta})=\delta$. Hence, we can apply Lemma \ref{shrinkage} to $\xi_{+},\xi_{-}:[t_{\delta},T]\to\mathbb{R}$, and in particular we have that
	\[G=d(T)\leq\kappa\delta.\]
	However, $\delta>0$ can be chosen arbitrarily small, and $\kappa$ is independent of our choice of $\delta$. Hence, we get that $\forall\delta>0:G\leq\kappa\delta\implies G\leq0$, which contradicts our assumption that $G>0$. Hence, it follows that $\xi_{+}(t)\leq\xi_{-}(t)$ for all $t$.
\end{proof}
As per the earlier remarks, this concludes the proof of Theorem \ref{L2 stability}. In particular, for any curve $\tilde{\xi}(t)$ lying between $\xi_{+}(t),\xi_{-}(t)$, we have that
\[\forall t\geq0:\pnorm{2}{u(\cdot,t)-\Phi(\cdot-\tilde{\xi}(t))}\leq\pnorm{2}{u_0-\Phi}.\]
Therefore, it suffices to consider e.g. $\overline{\xi}(t)=\frac{1}{2}\xi_{-}(t)+\frac{1}{2}\xi_{+}(t)-\sigma t$. Any continuous curve $\xi(t)$ such that $\xi(t)\in[\xi_{+}(t)-\sigma t,\xi_{-}(t)-\sigma t]$ for all $t\geq0$ will do the job, even if it is not Lipschitz. However, we can always make a choice of $\overline{\xi}$ that is Lipschitz.

\subsection{Magnitude of the shift}
In this section, we prove Corollary~\ref{magnitude} on the magnitude of the shift required to induce $L^2$-contraction. Here again we follow \cite{LegerShock} with allowances for heterogeneity. Let $\overline{\xi}(t)$ be as above. Now, the solutions to \eqref{final step} $\xi_{\pm}$ are uniformly Lipschitz with $\abs{\dot{\xi}_{\pm}}\leq L$, where $L$ depends only on $\supnorm{u_0},f,u_{\pm}$. We also know that $\xi_{\pm}(0)=0$, and $\xi_{+}(t)\leq\overline{\xi}(t)+\sigma t\leq\xi_{-}(t)$ which implies that
\[
\abs{\overline{\xi}(t)+\sigma t}\leq\max\left\{\abs{\xi_{+}(t)},\abs{\xi_{-}(t)}\right\}\leq Lt,
\]
and therefore
\[
\abs{\overline{\xi}(t)}\leq\left(L+\abs{\sigma}\right)t.
\]
Furthermore, we have that $\Phi(\cdot-\sigma t)-\Phi(\cdot-\sigma t-\overline{\xi}(t))$ is supported in
\[
\left[\min\{\sigma t,\sigma t+\overline{\xi}(t)\},\max\{\sigma t,\sigma t+\overline{\xi}(t)\}\right],
\]
which is contained in $\left[-\left(L+\abs{\sigma}\right)t,\left(L+\abs{\sigma}\right)t\right]$, since $\abs{\sigma t}\leq\abs{\sigma}t$ and $\abs{\sigma t+\overline{\xi}(t)}\leq Lt$. Now let
\[
V=\sup\left\{\abs{f_u(x,z)}:x\in\mathbb{R},\ \abs{z}\leq R\right\},\qquad R=\max\left\{\supnorm{u},\abs{u_{-}},\abs{u_{+}}\right\},
\]
which is well defined by \ref{FSP}, and bounds the relevant speeds of propagation for both $u$ and the (shifted) Riemann solutions. Therefore, we can say that
\[
\begin{split}
	(u_{-}-u_{+})\abs{\overline{\xi}(t)}=&\pnorm{L^1\left(-\left(L+\abs{\sigma}\right)t,\left(L+\abs{\sigma}\right)t\right)}{\Phi(\cdot-\sigma t)-\Phi(\cdot-\sigma t-\overline{\xi}(t))} \\
	\leq&\pnorm{L^1\left(-\left(L+\abs{\sigma}\right)t,\left(L+\abs{\sigma}\right)t\right)}{\Phi(\cdot-\sigma t)-u(\cdot,t)} \\ &+\pnorm{L^1\left(-\left(L+\abs{\sigma}\right)t,\left(L+\abs{\sigma}\right)t\right)}{u(\cdot,t)-\Phi(\cdot-\sigma t-\overline{\xi}(t))} \\
	\leq&\pnorm{L^1\left(-\left(L+\abs{\sigma}+V\right)t,\left(L+\abs{\sigma}+V\right)t\right)}{\Phi-u_0} \\
	&+\pnorm{L^1\left(-\left(L+\abs{\sigma}\right)t,\left(L+\abs{\sigma}\right)t\right)}{u(\cdot,t)-\Phi(\cdot-\sigma t-\overline{\xi}(t))},
\end{split}
\]
where the last inequality follows from Kružkov's local $L^1$-contraction/finite-propagation estimate for entropy solutions of scalar conservation laws \cite{Kruzkov}, with propagation speed $V$. Since $\Phi,u\in L^{\infty}$, we can apply Hölder's inequality to both terms on the right-hand side above and from Theorem \ref{L2 stability} obtain
\[
\begin{split}
	(u_{-}-u_{+})\abs{\overline{\xi}(t)}\leq&\sqrt{2\left(L+\abs{\sigma}+V\right)t}\pnorm{2}{\Phi-u_0} \\
	&+\sqrt{2\left(L+\abs{\sigma}\right)t}\pnorm{2}{u(\cdot,t)-\Phi(\cdot-\sigma t-\overline{\xi}(t))} \\
	\leq&\left(\sqrt{2\left(L+\abs{\sigma}\right)}+\sqrt{2\left(L+\abs{\sigma}+V\right)}\right)\pnorm{2}{u_0-\Phi}\sqrt{t},
\end{split}
\]
which proves the corollary.  Note that the shift curve $\overline{\xi}$ can be chosen such that it is Lipschitz, though it is not evident from this inequality. The only condition that the continuous curve $\overline{\xi}(t)$ must satisfy to induce $L^2$ contraction is that it must lie between the curves $\xi_{\pm}(t)-\sigma t$. If $\xi_{+}(t)\neq\xi_{-}(t)$, then in particular we can choose continuous curves that are not Lipschitz, but still induce $L^2$ contraction.

The constants $L,V$ are controlled by $\supnorm{u_0},f,u_{\pm}$ in the following way. By the existence of backward characteristics and the preservation of $f$ along genuine characteristics, we have that for all $t\geq0:\supnorm{f(\cdot,u(\cdot,t))}\leq\supnorm{f(\cdot,u_0(\cdot))}$. By \eqref{S} and \eqref{FSP}, $\supnorm{f(\cdot,u_0(\cdot))}$ is well-defined for $u_0\in L^{\infty}$. Now, by \eqref{Nag}, the uniform bound on $f(\cdot,u(\cdot,t))$ implies a uniform bound on $u(\cdot,t)$, for all $t>0$. Thus $L,V$ are both controlled by $\supnorm{u_0},f,u_{\pm}$ as claimed.

\subsection{Counterexample with negative heterogeneity}\label{neg het}
In this section, we justify the assumption \eqref{P} in Theorem~\ref{L2 stability}. In particular, we adapt the counterexample of section~\ref{ex} as follows: let $\varphi(x)$ now be a \textit{decreasing} rather than an increasing function, and let $u_{-}$ be the largest value of $u$ such that $h(u)=g(u)$. The function
\[
\eta(u)=\dfrac{(u-u_-)^2}{2}
\]
is a convex entropy, with the corresponding flux given by
\begin{align*}
	Q(x,u)&=\int_{u_-}^{u}(y-u_{-})\left[\varphi(x)h^{\prime}(y)+(1-\varphi(x))g^{\prime}(y)\right]dy \\
	\implies Q_x(x,u)&=\varphi^{\prime}(x)\int_{u_-}^{u}(y-u_{-})\left[h^{\prime}(y)-g^{\prime}(y)\right]dy.
\end{align*}
Let $\psi\in C_c^{\infty}(\mathbb{R})$ be a non-zero positive function. Then, for initial data $u_0(x)=u_{-}+\psi(x)$, a classical solution exists for at least some finite time, say $T$. In particular, the classical solution $u$ satisfies the entropy inequality as an \textit{equality}, i.e. in $\mathbb{R}\times[0,T],$ we have that
\begin{align*}
	\eta(u)_t+Q(x,u)_x&=Q_x(x,u)-\eta^{\prime}(u)f_x(x,u) \\
	&=\varphi^{\prime}(x)\left[\int_{u_{-}}^{u}(y-u_{-})[h^{\prime}(y)-g^{\prime}(y)]dy-(u-u_{-})\left[h(u)-g(u)\right]\right] \\
	&=-\varphi^{\prime}(x)\int_{u_-}^{u}\left[h(y)-g(y)\right]dy \\
	&\geq0.
\end{align*}
Indeed, by our choice of $u_{-}$, the inequality is strict wherever $u(x,t)\neq u_-$. Hence, integrating by parts over the domain $\mathbb{R}\times[0,T]$ yields
\[
\int_{\mathbb{R}}\dfrac{(u(x,T)-u_-)^2}{2}dx>\int_{\mathbb{R}}\dfrac{(u_0(x)-u_{-})^2}{2}dx,
\]
and thus the $L^2$ norm strictly increases. Note that in the homogeneous case, this is not possible. The non-divergence term in the entropy equation can be written for general $C^2$ fluxes with a stationary point $u_-$ as
\begin{align*}
	Q_x(x,u)-\eta^{\prime}(u)f_x(x,u)&=\int_{u_-}^{u}(y-u_{-})f_{xu}(x,y)dy-(u-u_-)f_x(x,u) \\
	&=-\int_{u_-}^{u}f_x(x,y)dy.
\end{align*}
Hence, if $f_{xu}(x_0,u_{-})<0$ for any $x_0\in\mathbb{R}$, we can always construct smooth initial data such that the $L^2$ distance of the entropy solution from the constant $u_-$ is strictly increasing for small enough times. Since $f\in C^2(\mathbb{R}^2)$ and $f_x(\cdot,u_-)=0$ by \eqref{S}, $f_{x}(x,u)<0$ for $u>u_{-}$  and $\abs{x-x_0},\abs{u-u_-}$ small enough. This also holds, mutatis mutandis, with any other stationary point $\overline{u}$ of the flux. Hence, by \eqref{FSP}, the simple shock solution of \eqref{claw} with initial data of the Riemann form \eqref{Riemann} is not contractive with respect to $L^2$ perturbations, even with shifts, if \eqref{P} is violated at $u_{\pm}$.

\section{Appendix: well-posedness}\label{app}
Since our assumptions \eqref{S}-\eqref{FSP} on the flux are distinct from those of Kružkov \cite{Kruzkov}, well-posedness of the conservation law \eqref{claw} is a subtle matter, though the same results (existence, uniqueness, and stability of solutions) hold. We note, however, that under the additional assumption \eqref{P}, the well-posedness results of Kružkov \cite{Kruzkov} do apply, where only the weaker assumption that $f_{xu}$ be bounded below is required.

It is sufficient to show that viscosity solutions to the Hamilton-Jacobi equation \eqref{HJ} exist for Lipschitz initial data, and that the solution is in turn uniformly Lipschitz, unique, and in particular obtained from the variational formulation. To prove this, we will work with the variational form of \eqref{HJ}, since it is known that viscosity solutions of \eqref{HJ} are obtained from the variational form for convex Hamiltonians \cite[Chapter 6]{CSHJB}.

\subsection{Variational formulation}
The Hamilton-Jacobi equations \eqref{HJ} can be equivalently analysed in variational form under the assumption of strict convexity via the Legendre transform, given by
\[
f^{\ast}(x,v)=\sup_{u\in\mathbb{R}}\{uv-f(x,u)\}.
\]
This is well-defined when $f$ is assumed to have super-linear growth, i.e. for all $x$:
\[
\lim_{\abs{u}\to\infty}\dfrac{f(x,u)}{\abs{u}}=+\infty,
\]
which is not a restrictive assumption when a priori $L^{\infty}$ estimates for solutions to the conservation law can be established. Here the role of assumption \eqref{Nag} is also made clear: it ensures that this limit is well-defined and holds, and thus that the supremum in the definition of the Legendre transform is a maximum over a compact set. For each $(x,t)\in\mathbb{R}\times[0,\infty)$, let $\mathscr{A}_{x,t}$ denote the set of admissible arcs given by
\[
\mathscr{A}_{x,t}=\{y\in W^{1,1}([0,t];\mathbb{R}):y(t)=x\},
\]
where $W^{1,1}$ denotes the standard Sobolev space. Then, a unique notion of solution to \eqref{HJ} with Lipschitz initial data can be obtained as the value function of the following cost minimisation problem \cite[pp. 161-165]{CSHJB}:
\[
v(x,t)=\inf_{y\in\mathscr{A}_{x,t}}\left\{v_0(y(0))+\int_{0}^{t}f^{\ast}(y(s),\dot{y}(s))ds\right\}.
\]
The infimum above can be shown to be a minimum that is achieved for specific (not necessarily unique) paths. It can also be shown that $v$ is Lipschitz and semiconcave, and that this semiconcavity is sufficient to characterise a unique notion of solution \cite[pp. 112-121]{CSHJB}. Therefore, $v$ satisfies \eqref{HJ} in a point-wise sense almost everywhere, and we can meaningfully speak of left and right spatial derivatives at all positive times, which in turn means we can make sense of traces for solutions of \eqref{claw}. Thus, we consider the following calculus of variations problem, for $v_0\in\operatorname{Lip}(\mathbb{R})$:
\begin{align}\label{CV}\tag{CV}
	v(x,t)=\inf_{y\in\mathscr{A}_{x,t}}\left\{v_0(y(0))+\int_{0}^{t}f^{\ast}(y(s),\dot{y}(s))ds\right\}
\end{align}
where $f^*$ denotes the Legendre transform, and $\mathcal{A}_{x,t}$ is, as defined before,
\[
\mathscr{A}_{x,t}=\{y\in W^{1,1}([0,t];\mathbb{R}):y(t)=x\},
\]
denoting the set of (absolutely continuous) admissible paths. We will demonstrate that minimising arcs exist, adapting the proof of \cite[Theorem 6.1.2]{CSHJB}, but relaxing the assumption that $v_0$ be bounded below; we only assume that it is uniformly Lipschitz. First we show that $f^{*}$ also satisfies a Nagumo-type superlinear growth condition. Let $u>u_{-}$, then
\begin{align*}
	f(x,u)&=f(u_-)+\int_{u_-}^uf_u(x,y)dy \\
	&\leq f(u_-)+\int_{u_-}^{u}\theta(y)dy.
\end{align*}
For $u>u_{-}$, let
\[
A(u)=f(u_-)+\int_{u_-}^{u}\theta(y)dy.
\]
Then $A$ has superlinear growth as $u\to+\infty$: by \ref{UC} and \ref{FSP}, for every $x$,
\[
A(u)\geq f(x,u)\geq f(u_{-})-\theta(u_{-})(u-u_{-})+\dfrac{\alpha}{2}(u-u_{-})^2,
\]
and the right-hand side is a fixed quadratic in $u$, independent of $x$. Similarly, we can define $A$ for $u<u_{+}$ by
\[
A(u)=f(u_+)+\int_{u}^{u_+}\theta(y)dy,
\]
so that
\begin{align*}
	f(x,u)&=f(u_+)+\int_{u_+}^{u}f_u(x,y)dy \\
	&\leq f(u_+)+\int_{u}^{u_+}\theta(y)dy \\
	&\leq A(u),\quad u<u_+.
\end{align*}
As before, $A(u)\geq f(x,u)\geq f(u_{+})-\theta(u_{+})(u_+-u)+\frac{\alpha}{2}(u-u_{+})^2$, so $A$ has superlinear growth as $u\to-\infty$ as well. For $u\in[u_{+},u_{-}]$, define $A(u)=f(u_{+})+\sigma(u-u_{+})$, where $\sigma$ is the Rankine-Hugoniot speed in \eqref{RH}. Hence, $f(x,u)\leq A(u)$, and $A$ has superlinear growth at $\pm\infty$. Thus, its Legendre transform $A^*$ is well-defined and convex with superlinear growth \cite[Chapter 3, Theorem 3]{evans}; note that $A$ itself need not be convex for these conclusions. Hence,
\begin{align*}
	f^*(x,v)&=\sup_{p\in\mathbb{R}}\left\{pv-f(x,p)\right\} \\
	&\geq\sup_{p\in\mathbb{R}}\left\{pv-A(p)\right\} \\
	&\geq A^*(v).
\end{align*}
Let $\overline{A}=\min_{y\in\mathbb{R}}\{A^*(y),0\}$ and let $M(y)=A^*(y)-\overline{A}$ so that $M\geq0$. Then, for any admissible arc, by Jensen's inequality:
\begin{equation*}
	\begin{split}
		\int_{0}^{t}f^{\ast}(y(s),\dot{y}(s))ds&\geq\int_{0}^{t}A^*(\dot{y}(s))ds \\
		&\geq\int_{0}^{t}M(\dot{y}(s))ds+t\overline{A} \\
		&\geq tM\left(t^{-1}\int_0^t\dot{y}(s)ds\right)+t\overline{A} \\
		&=t(M(t^{-1}(x-y(0)))+\overline{A}).
	\end{split}
\end{equation*}
Hence,
\begin{align*}
	v_0(y(0))+\int_{0}^{t}f^{\ast}(y(s),\dot{y}(s))ds&\geq v_0(y(0))+t(M(t^{-1}(x-y(0)))+\overline{A}).
\end{align*}
Now, $v_0$ is Lipschitz and thus grows at most linearly, while $M$ grows super-linearly as $y(0)\to\pm\infty$. Thus, the functional is bounded below, and we can meaningfully speak of minimising sequences in $\mathcal{A}_{x,t}$. Note that this also means $v(x,t)$ is well-defined and finite. Furthermore, if $y_k$ is such a minimising sequence, then $y_k(0)$ must be bounded. We want to show that $y_k$ contains a convergent subsequence. Then, by the lower semi-continuity of the functional, we obtain the existence of explicit minimising arcs.

By the Dunford-Pettis theorem, the compactness of $\{y_k\}$ is equivalent to showing that the sequence of derivatives $\dot{y}_k$ is equiabsolutely integrable, since one endpoint, $y_k(t)=x$, is fixed. Let $\varepsilon>0,\mu>0$. By superlinear growth, there exists $C_{\mu}$ such that $r\leq M(r)/\mu$ for all $r>C_{\mu}$. Now $M\geq0$, and
\[
\int_{0}^{t}M(\dot{y}_k(s))ds\leq-t\overline{A}+\int_{0}^{t}f^*(y_k(s),\dot{y}_k(s))ds.
\]
Suppose
\[
\sup_{k\in\mathbb{N}}{\abs{y_k(0)}}=Y_0, \quad  \max_{x\in[-Y_0,Y_0]}\abs{v_0(x)}=V_0.
\]
Then, for all $k\in\mathbb{N}$, we have that
\begin{align*}
	\int_{0}^{t}M(\dot{y}_k(s))ds&\leq-t\overline{A}+\sup_{k\in\mathbb{N}}\left\{v_0(y_k(0))+\int_{0}^{t}f^{\ast}(y_k(s),\dot{y}_k(s))ds\right\}+V_0,
\end{align*}
where the supremum is well-defined by our assumption that $\{y_k\}$ is a minimising sequence for the functional \eqref{CV}. Now, for any measurable set $E\subseteq[0,t]$, we have that
\begin{equation*}
	\begin{split}
		\int_{E}\abs{\dot{y_k}(s)}ds&\leq\dfrac{1}{\mu}\int_{E\cap\{\abs{\dot{y}_k}>C_\mu\}}M\left({\dot{y}_k(s)}\right)ds+\int_{E\cap\{\abs{\dot{y}_k}\leq C_\mu\}}\abs{\dot{y}_k(s)}ds \\
		&\leq\dfrac{1}{\mu}\tilde{C}+\abs{E}C_{\mu},
	\end{split}
\end{equation*}
where $\tilde{C}$ is some constant depending on $t,v_0$ and the minimising sequence; specifically
\[
\tilde{C}=-t\overline{A}+\sup_{k\in\mathbb{N}}\left\{v_0(y_k(0))+\int_{0}^{t}f^{\ast}(y_k(s),\dot{y}_k(s))ds\right\}+V_0.
\]
Hence, by taking $\mu$ large enough and $\abs{E}$ small enough, the integral on the LHS can be made smaller than $\epsilon$, thus proving that it is equiabsolutely integrable. Hence, a limit exists, which we denote by $y$. Without loss of generality, assume $y_k\to y$. From the lower semicontinuity of the functional, we conclude that $y$ is a minimiser \cite[Theorem 6.1.2]{CSHJB}.

The rest of the argument follows exactly as in \cite[Chapter 6]{CSHJB}, since the only assumption to be relaxed was the lower bound on $v_0$, which we have done by proving that the sequence $\{y_k(0)\}$ is bounded for any minimising sequence. The value function $v$ is the unique viscosity solution of the Hamilton-Jacobi equation \eqref{HJ}, as uniqueness follows from known results \cite{crandallions}, and $u=v_x$ is the corresponding unique entropy solution of \eqref{claw}.
\subsection{Generalising compact non-homogeneity}
To see the well-posedness of \eqref{claw} and its correspondence with \eqref{HJ}, we appeal to the compact non-homogeneity theory as laid out in \cite{thesis} and generalise it using the stability theorem and \eqref{FSP} assumption. In particular, for a flux satisfying assumptions \eqref{S}-\eqref{FSP}, let us consider a sequence of approximate fluxes defined by
\[
f_{K}(x,u)=f(x,u)\varphi_K(x)+(1-\varphi_K(x))g(u),
\]
where $\varphi_K$ is a compactly supported, smooth approximation by mollification of the characteristic function of $[-K-1,K+1]$ that is identically 1 on $[-K,K]$; in particular, consider
\[
\varphi_K(x)=\eta_{\epsilon}\ast\chi_{[-K-1,K+1]}(x),
\]
where $\eta_{\epsilon}$ is a standard symmetric mollifier supported in $[-\epsilon,\epsilon]$ with parameter $\epsilon<1/2$, and $g$ is defined as
\[
g(u)=\dfrac{\alpha}{2}(u-u_-)(u-u_+)+cu+d,
\]
where $c,d$ are constants chosen such that $g(u_{\pm})=f(u_{\pm})$. Then, $f_K$ satisfies all the properties \eqref{S}-\eqref{FSP} in addition to \eqref{CNH}, since $\partial_xf_K(x,u)=0$ for $\abs{x}>K+2$. Entropy solutions to \eqref{claw} (and, respectively, viscosity solutions for $\eqref{HJ}$) exist for these `approximate' fluxes. Now, by \eqref{FSP} and \cite[Theorem 8.2.3]{thesis}, we can conclude that the `approximate' solutions $u_K,u_J$ match on domains $[-N,N]\times[0,T]$ for any $N,T$ if $K,J$ are large enough relative to $N,T$. In particular, we have the stability estimate
\[
\int_{-N}^{N}\abs{u_K(x,T)-v_K(x,T)}dx\leq\int_{-N-LT}^{N+LT}\abs{u_0(x)-v_0(x)}dx
\]
for $K>N+LT$, where $u_K,v_K$ are entropy solutions of \eqref{claw} for flux $f_K$ and initial data $u_0,v_0\in L^{\infty}(\mathbb{R})$ respectively, and $L$ is independent of $K$ by \eqref{FSP} since
\begin{align*}
	L&=\sup_{x\in\mathbb{R},\abs{p}\leq\supnorm{u}+\supnorm{v}}\abs{\pdv{f_K}{u}(x,p)} \\
	&\leq\sup_{x\in\mathbb{R},\abs{p}\leq\supnorm{u}+\supnorm{v}}\abs{\pdv{f}{u}(x,p)}+\sup_{x\in\mathbb{R},\abs{p}\leq\supnorm{u}+\supnorm{v}}\abs{\pdv{g}{u}(p)} \\
	&\leq\theta\left(\supnorm{u}+\supnorm{v}\right)+(2\alpha+c)\left(\supnorm{u}+\supnorm{v}\right).
\end{align*}
Thus, the equation \eqref{claw} is well-posed in $L^{\infty}$ and correspondingly \eqref{HJ} is well-posed in $\operatorname{Lip}(\mathbb{R})$ \cite[Corollary 8.2.19]{thesis} (see also \cite[Remark 8.2.4]{thesis}). By correspondence with \eqref{HJ}, left and right spatial traces exist for the entropy solution at all positive times. Thus, Dafermos' theory of generalised characteristics \cite{GenChar} is applicable. Finally, we note that \eqref{S} and \eqref{UC} are not essential for \eqref{claw} to be well-posed; as long as the other conditions hold, \eqref{UC} can be relaxed to `strong convexity' in the sense of \cite[Definition 8.1.1]{thesis} and \eqref{S} can be relaxed to allow $u_{+}=u_{-}$, i.e. a single stationary state \cite{correspondence}. Alternatively, it is sufficient to assume that, in addition to \eqref{FSP}, $f(\cdot,u)$ is also uniformly bounded for each $u$.

\section{Acknowledgements}
The authors would like to thank the Department of Atomic Energy, Government of India, for their support under project no. 12-R\&D-TFR-5.01-0520.

\bibliographystyle{abbrv}
\bibliography{citations}

@article{HomoShock,
author = {Dafermos, Constantine and Shearer, Michael},
title = {FINITE TIME EMERGENCE OF A SHOCK WAVE FOR SCALAR CONSERVATION LAWS},
journal = {Journal of Hyperbolic Differential Equations},
volume = {07},
number = {01},
pages = {107-116},
year = {2010},
doi = {10.1142/S0219891610002037},
URL = {https://doi.org/10.1142/S0219891610002037},
eprint = {https://doi.org/10.1142/S0219891610002037}
}

@book{CSHJB,
	title={Semiconcave Functions, Hamilton-Jacobi Equations, and Optimal Control},
	url={http://dx.doi.org/10.1007/b138356}, 
	DOI={10.1007/b138356}, 
	publisher={Birkhäuser Boston}, 
	author={Cannarsa, Piermarco and Sinestrari, Carlo}, 
	year={2004} }

@phdthesis{thesis,
	TITLE = {{Heterogeneity in scalar conservation laws: approximation and applications}},
	AUTHOR = {Sylla, Abraham},
	URL = {https://hal.science/tel-03303049},
	SCHOOL = {{Universit{\'e} de Tours}},
	YEAR = {2021},
	MONTH = Jul,
	TYPE = {Theses},
	HAL_ID = {tel-03303049},
	HAL_VERSION = {v2},
}

@article{LegerShock,
	author={Leger, Nicholas},
	title={L$^2$ Stability Estimates for Shock Solutions of Scalar Conservation Laws Using the Relative Entropy Method},
	journal={Archive for Rational Mechanics and Analysis},
	year={2011},
	volume={199},
	number={3},
	pages={761-778},
	issn={1432-0673},
}

@article{kruzkov,
	doi = {10.1070/SM1970v010n02ABEH002156},
	url = {https://dx.doi.org/10.1070/SM1970v010n02ABEH002156},
	year = {1970},
	publisher = {},
	volume = {10},
	number = {2},
	pages = {217},
	author = {S. N. Kružkov},
	title = {FIRST ORDER QUASILINEAR EQUATIONS IN SEVERAL INDEPENDENT VARIABLES},
	journal = {Mathematics of the USSR-Sbornik}
}

@article{GenChar,
	ISSN = {00222518, 19435258},
	URL = {http://www.jstor.org/stable/24891601},
	author = {Dafermos, Constantine},
	journal = {Indiana University Mathematics Journal},
	number = {6},
	pages = {1097--1119},
	publisher = {Indiana University Mathematics Department},
	title = {Generalized Characteristics and the Structure of Solutions of Hyperbolic Conservation Laws},
	urldate = {2024-06-16},
	volume = {26},
	year = {1977}
}

@Article{GenCharSys,
	author={Dafermos, Constantine},
	title={Generalized characteristics in hyperbolic systems of conservation laws},
	journal={Archive for Rational Mechanics and Analysis},
	year={1989},
	volume={107},
	number={2},
	pages={127-155},
	issn={1432-0673},
	doi={10.1007/BF00286497},
	url={https://doi.org/10.1007/BF00286497}
}

@book{Filippov,
	title     = "Differential equations with discontinuous righthand sides",
	author    = "Filippov, A. F.",
	publisher = "Springer",
	series    = "Mathematics and its Applications",
	year      =  2010,
	address   = "Dordrecht, Netherlands",
	language  = "en"
}

@Article{ConLawHJB,
	author={Colombo, Rinaldo M.
	and Perrollaz, Vincent
	and Sylla, Abraham},
	title={Conservation laws and $\text{H}$amilton-$\text{J}$acobi equations with space inhomogeneity},
	journal={Journal of Evolution Equations},
	year={2023},
	volume={23},
	number={3},
	pages={50},
	issn={1424-3202},
	doi={10.1007/s00028-023-00902-1},
	url={https://doi.org/10.1007/s00028-023-00902-1}
}

@article{structuretheorem,
	author = {Adimurthi and Ghoshal, Shyam Sundar and Veerappa Gowda, G. D.},
	title = {STRUCTURE OF ENTROPY SOLUTIONS TO SCALAR CONSERVATION LAWS WITH STRICTLY CONVEX FLUX},
	journal = {Journal of Hyperbolic Differential Equations},
	volume = {09},
	number = {04},
	pages = {571-611},
	year = {2012},
	doi = {10.1142/S0219891612500191},
	URL = {https://doi.org/10.1142/S0219891612500191},
	eprint = {https://doi.org/10.1142/S0219891612500191}
}

@article{lp,
	title = {L$^p$ stability for entropy solutions of scalar conservation laws with strict convex flux},
	journal = {Journal of Differential Equations},
	volume = {256},
	number = {10},
	pages = {3395-3416},
	year = {2014},
	issn = {0022-0396},
	doi = {https://doi.org/10.1016/j.jde.2014.02.005},
	url = {https://www.sciencedirect.com/science/article/pii/S0022039614000783},
	author = { Adimurthi and Shyam Sundar Ghoshal and G.D. {Veerappa Gowda}}
}

@book{evans, 
	title={Partial Differential Equations},
	publisher={American Mathematical Society}, author={Evans, Lawrence C.}, year={2010} }

@article{exactcontrol,
	title = {Exact controllability of scalar conservation laws with strict convex flux},
	author = {Adimurthi and Ghoshal, Shyam Sundar and Veerappa Gowda, G. D.},
	journal = {Mathematical Control and Related Fields},
	volume = {4},
	number = {4},
	pages = {401-449},
	year = {2014},
	issn = {2156-8472},
	doi = {10.3934/mcrf.2014.4.401},
	url = {https://www.aimsciences.org/article/id/13c390a1-b231-4d9b-9875-005b39172fcb}
}

@article{bvuni,
	author={Chen, Geng
	and Krupa, Sam G.
	and Vasseur, Alexis F.},
	title={Uniqueness and Weak-$\text{BV}$ Stability for $2\times2$ Conservation Laws},
	journal={Archive for Rational Mechanics and Analysis},
	year={2022},
	volume={246},
	number={1},
	pages={299-332},
	issn={1432-0673},
	doi={10.1007/s00205-022-01813-0},
	url={https://doi.org/10.1007/s00205-022-01813-0}
}

@article{sharpL2contra,
	title={Sharp a-contraction estimates for small extremal shocks},
	author={William M. Golding and Sam G. Krupa and Alexis F. Vasseur},
	journal={Journal of Hyperbolic Differential Equations},
    volume={20},
    number={3},
    pages={541–602},
	year={2023},
	url={https://api.semanticscholar.org/CorpusID:229298092}
}

@article{criteriacontra,
	title={Criteria on Contractions for Entropic Discontinuities of Systems of Conservation Laws},
	author={Moon-Jin Kang and Alexis F. Vasseur},
	journal={Archive for Rational Mechanics and Analysis},
	year={2015},
	volume={222},
	pages={343-391},
	url={https://api.semanticscholar.org/CorpusID:37516445}
}

@inbook{lwr,
	publisher = {John Wiley \& Sons, Ltd},
	author = {Whitham, G. B.},
	title = {Linear and Nonlinear Waves},
	chapter = {3},
	pages = {68-95},
	doi = {https://doi.org/10.1002/9781118032954.ch3},
	url = {https://onlinelibrary.wiley.com/doi/abs/10.1002/9781118032954.ch3},
	eprint = {https://onlinelibrary.wiley.com/doi/pdf/10.1002/9781118032954.ch3},
	year = {1999}
}

@article{trace,
	author = {Neves, Wladimir and Panov, Evgeniy and Silva, Jean},
	title = {Strong Traces for Conservation Laws with General Nonautonomous Flux},
	journal = {SIAM Journal on Mathematical Analysis},
	volume = {50},
	number = {6},
	pages = {6049-6081},
	year = {2018},
	doi = {10.1137/17M1159828},
	URL = {https://doi.org/10.1137/17M1159828},
	eprint = { https://doi.org/10.1137/17M1159828}
}

@article{discofluxog,
	title={Conservation law with discontinuous flux},
	author={Adimurthi and G. D. Veerappa Gowda},
	journal={Journal of Mathematics of Kyoto University},
	year={2003},
	volume={43},
	pages={27-70},
	url={https://api.semanticscholar.org/CorpusID:122580940}
}

@article{correspondence,
	title = {Scalar conservation laws and $\text{H}$amilton-$\text{J}$acobi equations in one-space variable},
	journal = {Nonlinear Analysis: Theory, Methods and Applications},
	volume = {18},
	number = {5},
	pages = {461-469},
	year = {1992},
	issn = {0362-546X},
	doi = {https://doi.org/10.1016/0362-546X(92)90013-5},
	url = {https://www.sciencedirect.com/science/article/pii/0362546X92900135},
	author = {Vincent Caselles}
}

@article{crandallions,
	title = {On existence and uniqueness of solutions of $\text{H}$amilton-$\text{J}$acobi equations},
	journal = {Nonlinear Analysis: Theory, Methods and Applications},
	volume = {10},
	number = {4},
	pages = {353-370},
	year = {1986},
	issn = {0362-546X},
	doi = {https://doi.org/10.1016/0362-546X(86)90133-1},
	url = {https://www.sciencedirect.com/science/article/pii/0362546X86901331},
	author = {Michael G. Crandall and Pierre-Louis Lions}
}
\end{document}